\documentclass[a4paper, 12pt]{article}

\usepackage{amsfonts,bm,amsmath,amssymb,amsthm}
\usepackage[english]{babel}
\usepackage[koi8-r]{inputenc}
\usepackage{epsfig,color, graphicx,psfrag}
\usepackage{wrapfig}
\usepackage{voyager}

\newcommand{\bbR}{\mathbb{R}}

\newcommand{\bbZ}{\mathbb{Z}}

\newcommand{\eps}{\varepsilon}
\renewcommand{\setminus}{\backslash}

\makeatletter
\def\@sect#1#2#3#4#5#6[#7]#8{\ifnum #2>\c@secnumdepth
     \let\@svsec\@empty\else
     \refstepcounter{#1}\edef\@svsec{%
     \csname pre#1\endcsname 
     \csname the#1\endcsname
     \csname post#1\endcsname 
\hskip 1em}\fi
     \@tempskipa #5\relax
      \ifdim \@tempskipa>\z@
	\begingroup #6\relax
	  \@hangfrom{\hskip #3\relax\@svsec
		   }{\interlinepenalty \@M \ignorespaces#8\par}%
	\endgroup
       \csname #1mark\endcsname{\protect\ignorespaces #7}\addcontentsline
	 {toc}{#1}{\ifnum #2>\c@secnumdepth \else
		      \protect\numberline{\csname the#1\endcsname
		     \csname post#1\endcsname 
}\fi
\protect\ignorespaces #7}\else
	\def\@svsechd{#6\hskip #3\relax  
		   \@svsec #8\csname #1mark\endcsname
		      {\protect\ignorespaces #7}\addcontentsline
			   {toc}{#1}{\ifnum #2>\c@secnumdepth \else
			     \protect\numberline{\csname the#1\endcsname
			  \csname post#1\endcsname 
}\fi
\protect\ignorespaces #7}}\fi
     \@xsect{#5}}
\makeatother

\newtheorem{theorem}{Theorem}[section]
\newtheorem{lemma}[theorem]{Lemma}
\newtheorem{proposition}[theorem]{Proposition}
\newtheorem{corollary}[theorem]{Corrollary}
\theoremstyle{definition}
\newtheorem{definition}[theorem]{Defenition}
\newtheorem{remark}[theorem]{Remark}

\newcommand{\fr}{\frac} 

\newcommand{\dtx}{\dot{x}}
\newcommand{\dty}{\dot{y}}
\newcommand{\padi}[2]{\fr{\partial #1}{\partial #2}}
\newcommand{\dd}[2]{\frac{d #1}{d #2}}

\newcommand{\bref}[1]{(\ref{#1})}

\newcommand{\dt}{\dot}
\newcommand{\Sepl}{S_\eps^+}
\newcommand{\Semn}{S_\eps^-}

\newcommand{\const}{\mathop{\rm const}}
\newcommand{\bbT}{\mathbb T}
\newcommand{\gaeps}{\gamma_\eps}

\newcommand{\delp}{\delta_{+}}

\newcommand{\delpl}{\delta_{+}}
\newcommand{\delmn}{\delta_{-}}
\newcommand{\delpm}{\delta_{\pm}}

\newcommand{\xpl}{x_{+}}
\newcommand{\xmn}{x_{-}}
\newcommand{\xpm}{x_{\pm}}

\newcommand{\Upl}{U^{+}}
\newcommand{\Umn}{U^{-}}
\newcommand{\Upm}{U^{\pm}}

\newcommand{\Dep}{D_\eps^+}
\newcommand{\Dem}{D_\eps^-}
\newcommand{\Depm}{D_\eps^\pm}

\newcommand{\Depl}{\Dep}

\newcommand{\Pep}{\Pi_\eps^{+}}
\newcommand{\Pem}{\Pi_\eps^{-}}
\newcommand{\spl}{\sigma^{+}}
\newcommand{\tpl}{\tau^{+}}
\newcommand{\smn}{\sigma^{-}}
\newcommand{\tmn}{\tau^{-}}
\newcommand{\spm}{\sigma^{\pm}}
\newcommand{\tpm}{\tau^{\pm}}
\newcommand{\Gapl}{\Gamma^{+}}
\newcommand{\Gamn}{\Gamma^{-}}
\newcommand{\Gapm}{\Gamma^{\pm}}
\newcommand{\Gpl}{G^{+}}
\newcommand{\Gmn}{G^{-}}

\newcommand{\Jpl}{J^{+}}
\newcommand{\Jmn}{J^{-}}

\newcommand{\rPeab}[2]{P_\eps^{\rarc{#1}{#2}}}
\newcommand{\lPeab}[2]{P_\eps^{\larc{#2}{#1}}}
\newcommand{\alpl}{\alpha^{+}}
\newcommand{\almn}{\alpha^{-}}
\newcommand{\alpm}{\alpha^{\pm}}

\newcommand{\pepl}{p^{+}_\eps}
\newcommand{\pemn}{p^{-}_\eps}
\newcommand{\pepm}{p^{\pm}_\eps}
\newcommand{\qepl}{q^{+}_\eps}
\newcommand{\qemn}{q^{-}_\eps}
\newcommand{\qepm}{q^{\pm}_\eps}
\newcommand{\Bepl}{B_{\eps}^{+}}
\newcommand{\Bemn}{B_{\eps}^{-}}
\newcommand{\Bepm}{B_{\eps}^{\pm}}
\newcommand{\Aepl}{A_{\eps}^{+}}
\newcommand{\Aemn}{A_{\eps}^{-}}
\newcommand{\Aepm}{A_{\eps}^{\pm}}
\newcommand{\Eepl}{E_{\eps}^{+}}
\newcommand{\Eemn}{E_{\eps}^{-}}
\newcommand{\Eepm}{E_{\eps}^{\pm}}

\newcommand{\Beplb}{\bar B_{\eps}^{+}}
\newcommand{\Bemnb}{\bar B_{\eps}^{-}}

\newcommand{\Aeplb}{\bar A_{\eps}^{+}}
\newcommand{\Aemnb}{\bar A_{\eps}^{-}}
\newcommand{\Aepmb}{\bar A_{\eps}^{\pm}}
\newcommand{\Eeplb}{\bar E_{\eps}^{+}}
\newcommand{\Eemnb}{\bar E_{\eps}^{-}}
\newcommand{\Eepmb}{\bar E_{\eps}^{\pm}}

\newcommand{\apl}{a^{+}}
\newcommand{\amn}{a^{-}}
\newcommand{\apm}{a^{\pm}}

\newcommand{\ypl}{y^{+}}
\newcommand{\ymn}{y^{-}}
\newcommand{\ypm}{y^{\pm}}

\newcommand{\Cnpl}{C_n^{+}}
\newcommand{\Cnmn}{C_n^{-}}
\newcommand{\Cnpm}{C_n^{\pm}}

\newcommand{\tCnpl}{\widetilde{C}_n^{+}}
\newcommand{\tCnmn}{\widetilde{C}_n^{-}}
\newcommand{\tCnpm}{\widetilde{C}_n^{\pm}}
\newcommand{\dde}{\frac{d}{d\eps}}

\newcommand{\Phpl}{\Phi^{+}}
\newcommand{\Phmn}{\Phi^{-}}
\newcommand{\Phpm}{\Phi^{\pm}}

\newcommand{\yptom}{y_{+}^{-}}

\newcommand{\Qeinv}{Q_\eps^{-1}}
\newcommand{\laeps}{\lambda_\eps}
\newcommand{\Ga}{\Gamma}

\newcommand{\Gaeps}{\Gamma_\eps}
\newcommand{\la}{\lambda}
\newcommand{\inv}[1]{\left(#1\right)^{-1}}
\newcommand{\wtQ}{\widetilde{Q}}

\newcommand{\Ppm}{P_{+}^{-}}
\newcommand{\Si}{\Sigma}

\newcommand{\tgam}{\tilde{\gamma}}
\newcommand{\yy}{y'_{y_0}}

\newcommand{\tix}{\tilde{x}}
\newcommand{\Reps}{R_\eps}
\newcommand{\Qeps}{Q_\eps}
\newcommand{\rarc}[2]{[#1,#2\rangle}
\newcommand{\larc}[2]{\langle#1,#2]}
\newcommand{\Repl}{\Reps^+}
\newcommand{\Remn}{\Reps^-}

\newcommand{\bRepl}{\bar{R}_\eps^+}
\newcommand{\bRemn}{\bar{R}_\eps^-}
\newcommand{\bRepm}{\bar{R}_\eps^{\pm}}

\title{Ducks on the torus: existence and uniqueness}
\author{Ilya V. Schurov\thanks{Moscow State University, Mechanics and
mathematics department, Leninskie gory, Moscow, 119991. E-mail:
ilya@schurov.com. The author is supported in part by RFBR and CNRS (project
05-01-02801-CNRS-L-a) and RFBR project 07-01-00017-a}}
\date{}

\begin{document}

\sloppy
\maketitle




\begin{abstract}
	We show that there exist generic slow-fast systems with only one
	(time-scaling) parameter on the two-torus, which have canard cycles for
	arbitrary small values of this parameter. This is in drastic contrast with
	the planar case, where canards usually occur in two-parametric families. Here
	we treat systems with a convex slow curve. In this case there is a set of
	parameter values accumulating to zero for which the system has exactly one
	attracting and one repelling canard cycle. The basin of the attracting cycle is
	almost the whole torus.
\end{abstract}
\vskip 2pc

\hskip 1.3pc UDC 517.925.42+517.938.

\vskip 1pc


\vskip 2pc

\section{Introduction}

Consider a generic slow-fast system on the plane:
\begin{equation}
	\begin{cases}
		\dt x=f(x,y,\eps)\\
		\dt y=\eps g(x,y,\eps)\\
	\end{cases}\quad (x,y)\in\bbR^2,\quad \eps\in(\bbR,0).
\end{equation}
There is a rather simple description of its behavior for small $\eps$. It consists of
interchanging phases of slow motion along stable parts of the slow curve
$M:=\{(x,y)\mid f(x,y,0)=0\}$ and fast jumps along straight lines $y=\const$. Given
additional parameters, which depend on $\eps$, one can observe more complicated
behavior: appearance of \emph{duck} (or \emph{canard}) solutions (particularly
limit cycles), i.e. solutions, whose phase curves contain an arc of length
bounded away from 0 uniformly in $\eps$, that keeps close to the unstable part
of the slow curve (see~\cite{D} and~\cite{DR}). 


In~\cite{GI}, Yu.~S.~Ilyashenko and J.~Guckenheimer discovered
a new kind of behavior of slow-fast systems on the two-torus. It
was shown that for some particluar family with no auxiliary parameters
there exists a sequence of intervals accumulating at $0$, such that for any
$\eps$ from these intervals, the system has exactly two limit cycles,
both of which are canards, where one is stable and the other unstable.

Yu.~S.~Ilyashenko and J.~Guckenheimer conjectured that there exists an
open domain in the space of slow-fast systems on the two-torus with the same
property. This work is devoted to the proof of this conjecture: we generalize
the result of~\cite{GI} for generic slow-fast systems with convex slow
curve. The work is based on the ideas of~\cite{GI} and has similar structure:
Section 2 states the Main Theorem and  outlines its proof,
consisting of a sequence of auxiliary lemmas. In Section 3 we state necessary
theorems about normal forms of slow-fast systems. The lemmas are proved in Section
4, and some auxiliary propositions are proved in the Appendix (Section 5).

The author would like to express his sincere appreciation to Yu.~S.~Ilyashenko
for the statement of the problem and his assistance with the work, to A.~Fishkin
for assistance with the work, to V.~Kleptsyn for fruitful discussions, valuable
comments on the text of the work and idea of using Distortion Lemma, to
G.~Kolutsky for valuable comments on the text of the work. The author also
grateful to the anonymous referee for valuable comments.

\section{Slow-fast systems on the two-torus and Poincar\'e map}
\subsection{Preliminary statement of the main result}
Consider a slow-fast system on the two-torus:
\begin{equation}\label{eq-main}
	\begin{cases}
		\dt x=f(x,y,\eps)\\
	\dt y=\eps g(x,y,\eps)\\
	\end{cases}\quad (x,y)\in\bbT^2\cong \bbR^2/(2\pi\bbZ^2),\ \eps \in (\bbR,0),
\end{equation}
where functions $f$ and $g$ are assumed to be smooth enough.

The following theorem is a corollary of the main result (Theorem~\ref{thm-main}
below):

\begin{theorem}
	\label{thm-main-euristic}

	There exists an open set in the space of slow-fast systems on the two-torus
	with the following property. For every system from this set there exists a
	sequence of intervals accumulating at zero, such that for every $\eps$ that belongs
	to these intervals the system has an attracting canard cycle. The basin of this
	cycle is the whole torus excluding exactly one unstable cycle.
\end{theorem}

A rigorous definition of the term ``canard solution'' as well as conditions that
define the open set mentioned in the theorem above are given in the next
section.  The main result (Theorem ~\ref{thm-main}) is a stronger version of
Theorem~\ref{thm-main-euristic}.

\subsection{Full statement of the main result}
For the slow-fast system~\bref{eq-main} denote its slow curve by $M$:
\begin{equation}
	M:=\{(x,y)\mid f(x,y,0)=0\}
\end{equation}
Impose the following conditions of local genericity on system~\bref{eq-main}:

\begin{enumerate}
	\item\label{enum-cond-first} The speed of the slow motion is bounded away from zero: $g>0$.
	\item $M$ is a smooth curve. 

\item The lift of the curve $M$ to the covering coordinate plane is contained
		in the interior of the fundamental square $\{|x|<\pi,\ |y|<\pi\}$ and is
		convex. This, in particular, implies that there are two jump points
		(straight and inverse jumps), which are the far right and the far left
		points of $M$. (See Fig.~\ref{fig-general-view}.) We denote them $G^-$ and
		$G^+$ respectively.

\item The following nondegenericity assumption holds in every point $(x,y)\in
	M\setminus\{G^+,G^-\}$:
	\begin{equation}\label{eq-nondeg-nondeg}
		\padi{f(x,y,0)}{x}\ne 0.
	\end{equation}

	\item \label{enum-cond-last} The following nondegenericity assumptions hold
		in the jump points: 
		\begin{equation}\label{eq-nondeg-main}
			\left.\padi{^2 f(x,y,0)}{x^2}\right|_{G^\pm} \ne 0,\quad
			\left.\padi{f(x,y,0)}{y}\right|_{G^\pm} \ne 0
		\end{equation}
\end{enumerate}
Conditions \ref{enum-cond-first}--\ref{enum-cond-last} define an open set in the
space of slow-fast systems on the two-torus.

\begin{figure}[t]
	\psfrag{Jplus}{$\Jpl$}
	\psfrag{Gaminus}{$\Gamn$}
	\psfrag{Gaplus}{$\Gapl$}
	\psfrag{x}{$x$}
	\psfrag{y}{$y$}
	\psfrag{Dpl}{$\Dep$}
	\psfrag{Gpl}{$\Gpl$}
	\psfrag{Mpl}{$M^+$}
	\psfrag{Mmn}{$M^-$}
	\psfrag{Gmn}{$\Gmn$}
	\psfrag{Jminus}{$\Jmn$}
	\psfrag{Dmn}{$\Dem$}
	\psfrag{Dmn}{$\Dem$}
	\psfrag{delpl}{$\delpl$}
	\psfrag{delmn}{$\delmn$}
	\psfrag{pi}{$\pi$}
	\psfrag{mp}{$-\pi$}
	\psfrag{tpl}{$\tpl$}
	\psfrag{alpl}{$\alpl$}
	\psfrag{almn}{$\almn$}
	\psfrag{tmn}{$\tmn$}
	\psfrag{Gamma}{$\Gamma$}
	\begin{center}
		\includegraphics[scale=1]{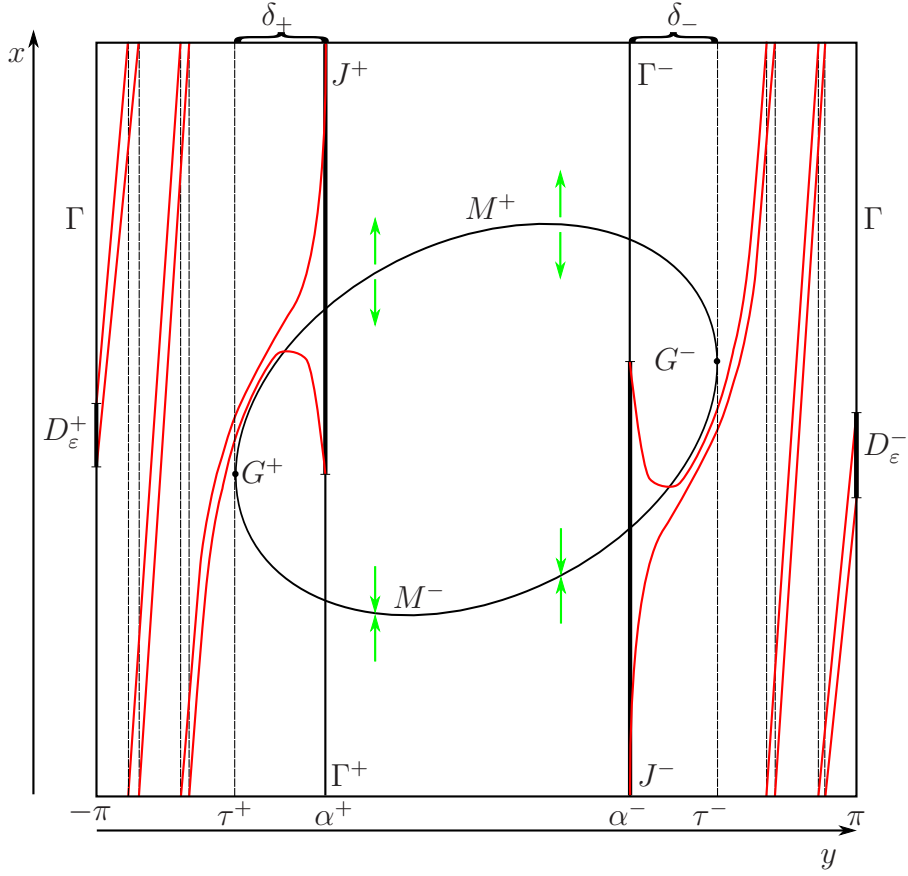}
		\caption{Slow-fast system on the two-torus. $x$-axis is directed upward, red lines are trajectories of the system.}\label{fig-general-view}
	\end{center}
\end{figure}
\begin{remark}
	Without loss of generality (by altering the direction of the $x$-axis if necessary)
	we can assume that: 
	\begin{equation}\label{eq-nondeg-main-wlog}
		\begin{array}{rcl}
			\left.\padi{f(x,y,0)}{y}\right|_{G^+} < 0,&\quad&
			\left.\padi{^2 f(x,y,0)}{x^2}\right|_{G^+} > 0,\\
			&\vspace{0pt}&\\
			\left.\padi{f(x,y,0)}{y}\right|_{G^-} > 0,&\quad&
			\left.\padi{^2 f(x,y,0)}{x^2}\right|_{G^-} > 0.
		\end{array}
	\end{equation}
\end{remark}
Let $M^\pm$ be the stable ($-$) and unstable ($+$) parts of the slow curve.  Let us
fix a vertical interval $J^+$ that crosses the unstable part of the slow curve
$M^+$ near the jump point $G^+$ and does not intersect $M^-$(the exact position of this interval will be
specified later). 

\begin{definition}\label{def-canards}
Any solution that crosses $J^+$ is a \emph{canard solution}.
\end{definition}

In order to construct canard cycles, we will study the Poincar\'e map $P_\eps$
from the global cross-section $\Gamma=\{y=-\pi\}$ to itself along the
trajectories of system~\bref{eq-main}. As the function $g$ is bounded away from
zero, this map is a well-defined diffeomorphism of a circle to itself. Note that
periodic trajectories of system~\bref{eq-main} correspond to periodic (in
particular, fixed) points of this map; thus, it is natural to consider the
rotation number of~$P_\eps$. Let us lift~$P_\eps$ to universal cover and denote
the rotation number of lifted map by $\rho(\eps)\in \bbR$.

\begin{remark}
	Hereafter by ``the assertion holds for small $\eps$'' we shall mean that the
	following condition is satisfied: there exists an $\eps_0>0$ such that the
	assertion holds for every $\eps\in (0,\eps_0]$. We can choose a universal
	$\eps_0$ for all such assertions in the paper.
\end{remark}

\begin{theorem}[Main result]
	\label{thm-main}
	For any system~\bref{eq-main} that satisfies conditions
	\ref{enum-cond-first}--\ref{enum-cond-last}, there exists a sequence of
	non-intersecting intervals $\{R_n\}$, $R_n=[\alpha_n,\beta_n]$ and two
	sequences of non-intersecting intervals $C_n^\pm \subset R_n$ with the
	following properties:
	\begin{enumerate}
			\item $|R_n|=O(e^{-Cn})$ for some $C>0$.
			\item $\alpha_n=O(1/n)$
				\item For every $\eps$ sufficiently small, not belonging to the
					$R_n$'s, the rotation number $\rho(\eps)$ is an integer. For
					such~$\eps$ there are exactly two periodic trajectories, both of
					which are hyperbolic, where one is stable and the other unstable.
					The unstable one is a canard. 
					\item For every sufficiently small $\eps\in C_n^\pm$,
						system~\bref{eq-main} has exactly two periodic trajectories,
						both of which are hyperbolic (where one stable and one
						unstable), and both are canards.
				\end{enumerate}
\end{theorem}

\begin{remark}
	The condition of convexity of~$M$ above can be weakened to that of the
	existence of exactly two points of~$M$ with a vertical tangent line. Indeed,
	using a smooth coordinate change that preserves vertical circles, any such
	curve can be made convex. On the other hand, such a change will not affect
	the conditions of local genericity.
\end{remark}
As it was mentioned above, the main tool of our study will be the Poincar\'e
map~$P_\eps$. In Section~\ref{ssec-poincare-map} three Lemmas (\ref{lem-shape},
\ref{lem-monot} and \ref{lem-conv}) are stated. They describe the behavior of the
Poincar\'e map as $\eps\to 0^+$. In section~\ref{ssec-canard-existence}, the Main
Theorem~\ref{thm-main} is proved modulo these Lemmas.

\subsection{Poincar\'e map}\label{ssec-poincare-map}
\begin{remark}\label{rem-dividing-out-g}

Note that the function $g$ is bounded away from zero, so we can divide the
system~\eqref{eq-main} by it, thus re-scaling the time: this does not change the
desired properties of its solutions, and the system with new function $f$ still
belongs to the same open set. Thus without loss of generality we can assume
$g=1$ in~\bref{eq-main} and consider the system:

\begin{equation}\label{eq-main-norm}
	\begin{cases}
		\dtx=f(x,y,\eps)\\
		\dty=\eps\\
	\end{cases}\quad (x,y)\in\bbT,\ \eps\in (\bbR,0)
\end{equation}
\end{remark}
Denote the graph of the Poincar\'e map $P_\eps$ by $\gaeps\subset S^1\times S^1$.
The following lemma shows that this graph looks more and more like the union of a
horizontal and a vertical circle as $\eps$ tends to $0^+$. In other words, the
derivative of the Poincar\'e map is (exponentially) small on the whole $\Gamma$
with the exception of an exponentially small interval.

\begin{lemma}[Shape Lemma]
\label{lem-shape}
There exist constants $c^{\pm}_{1,2}>0$ such that for any sufficiently small
${\eps>0}$ one can find two intervals $\Dep$ and $\Dem$ in the preimage and
image of $P_\eps$ resp. with the following properties: 

\begin{enumerate}
	\item $|\Depm|=O(e^{-c_1^{\pm}/\eps})$

	\item $\left.|P_\eps '|\right|_{S^1\setminus \Dep} = O(e^{-c_2^{+}/\eps})$

	\item $ \left.|(P_\eps^{-1})'|\right|_{S^1\setminus \Dem} = O(e^{-c_2^{-}/\eps})$
	\item The graph $\gaeps$ lies in the union of two orthogonal rings: $\Pep:=\Dep \times S^1$ and $\Pem:=S^1 \times \Dem$.
\end{enumerate}
\end{lemma}

Lemma~\ref{lem-shape} is proved in Section~\ref{ssec-shape}. The next lemma
formalizes the following observation: the lift  of $\gaeps$ to the fundamental
domain $\{|x|<\pi, |y|<\pi\}$ moves monotonically to the upper left as $\eps\to
0^+$. To formalize this, we need to introduce some additional notation.

Consider arbitrary points~$a$ and~$b$ on the oriented circle~$S^1$. They split
the circle into two arcs. Denote the arc from point $a$ to point $b$ (in the
sense of the orientation of the circle) by~$\rarc{a}{b}$. The orientation of this
arc is induced by the orientation of the circle. Also denote the same arc with
the inversed orientation by~$\larc{a}{b}$. (See fig.~\ref{fig-arcs}.)

\begin{figure}[htb]
	\psfrag{a}{$a$}
	\psfrag{b}{$b$}
	\psfrag{S1}{$S^1$}
	\psfrag{larcab}{$\larc{a}{b}$}
	\psfrag{rarcab}{$\rarc{a}{b}$}
	\psfrag{larcba}{$\larc{b}{a}$}
	\psfrag{rarcba}{$\rarc{b}{a}$}

	\begin{center}
		\includegraphics{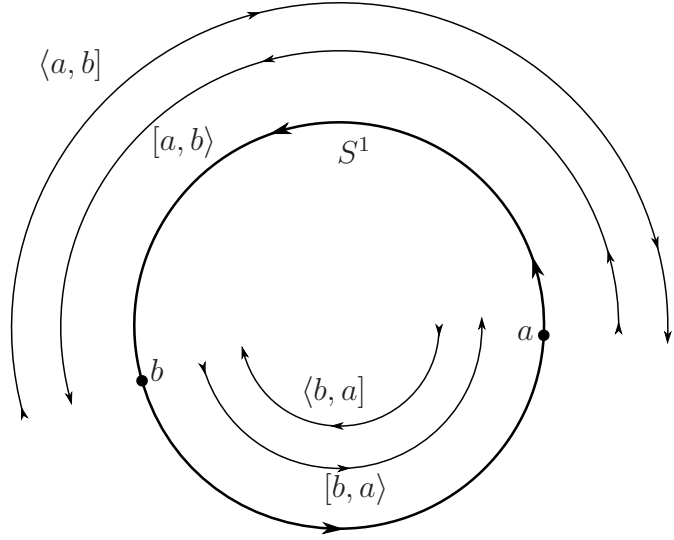}
		\caption{Orientation of the arcs}\label{fig-arcs}
	\end{center}
\end{figure}

Denote the Poincar\'e map along the phase curves of the main system~\bref{eq-main}
from the cross-section~$y=a$ to the cross-section~$y=b$ in the forward time (i.e.
along the arc $\rarc{a}{b}$) by $\rPeab{a}{b}$. Also, let
$\lPeab{b}{a}=\inv{\rPeab{a}{b}}$: this is the Poincar\'e map from the
cross-section $y=b$ to the cross-section $y=a$ in the reverse time. This fact is
stressed by the notation: the direction of the angle bracket shows the time
direction.

Hereafter any formula containing a $\pm$ or $\mp$ sign replaces two formulae:
one with all the upper sign and another with all the lower sign.

Let the jump points $G^\pm$ have coordinates $(\spm,\tpm)$ resp., (the slow
curve $M$ lies inside the strip~$\{\tpl\le y\le \tmn\}$ due to its convexity,
see Fig.~\ref{fig-general-view}).  For some fixed small positive $\delpl$ and
$\delmn$ we define the following objects:

\begin{enumerate}
		\item Transversal cross-sections $\Gapm$, intersecting the slow curve $M$
			near the jump points:
		\begin{equation*}
			\Gapm:=S^{1} \times \{\alpm\},
		\end{equation*}
		where $\alpm=\tpm \pm \delpm$. (See Fig.~\ref{fig-general-view}.)
		\item Segments $\Jpl$ and $\Jmn$ of the cross-sections $\Gapm$, which
			intersect unstable and stable parts of the slow curve resp.:
		\begin{gather*}
			\Jpl:=\{(x,\alpl) \in \Gapl \mid x \in [\spl,\pi]\},\\
			\Jmn:=\{(x,\almn) \in \Gamn \mid x \in [-\pi, \smn]\}.\\
		\end{gather*}
	\item Segments $\Depm$ (whose existence is provided by the Shape
		Lemma~\ref{lem-shape}):
		\begin{gather*}
			\Dep:=\lPeab{\alpl\!}{-\pi}(\Jpl), \quad \Dem:=\rPeab{\almn\!}{\,\pi}(\Jmn),\quad \Depm=\rarc{\pepm}{\qepm},
		\end{gather*}
\end{enumerate}

The choice of segment~$\Jpl$ here formally completes Definition~\ref{def-canards}
that used this interval. Note that all trajectories that intersect segment $\Dep$ also 
intersect $\Jpl$ and thus are canards.

Let $\Bepl$ and $\Bemn$ be the points on the graph $\gaeps$, where the slope of
$P_\eps$ is equal to $1$. As it will be shown below (see
Lemma~\ref{lem-conv}) there are exactly two such points. From the Lemma
~\ref{lem-shape} it follows that the points $\Bepm$ lie in the rectangle
$K_\eps:=\Dep\times\Dem$, because outside of this rectangle the derivative of
$P_\eps$ is either very small or very big. For the sake of
definiteness, we will assume that 
$\rarc{x(\Bemn)}{x(\Bepl)}\subset \Dep$: let us denote this as
$x(\Bemn)<x(\Bepl)$. 

\begin{remark}\label{rem-B-relations}
The map $P_\eps$ preserves orientation (i.e. monotonic) and is bijective, so
$\rarc{y(\Bepl)}{y(\Bemn})\subset \Dem$ (i.e.  $y(\Bemn)>y(\Bepl)$).
\end{remark}
Let us denote by $\Aepl$ and $\Aemn$ the points of the graph $\gaeps$ which lie
above the ends of the segment $\Depl$ (see figure~\ref{fig-poincare-graph}):

\begin{equation*}
	\Aemn:=(\pepl,P_\eps(\pepl)),\quad \Aepl:=(\qepl,P_\eps(\qepl)),
\end{equation*}
Denote also the top left and bottom right (in the sense of orientations of
coordinate circles) corners of the rectangle $K_\eps$ by $\Eemn$ and $\Eepl$
resp.:

\begin{equation*}
	\Eemn:=(\pepl,\qemn),\quad \Eepl:=(\qepl,\pemn)
\end{equation*}

Let $C_\eps$ stand for either of the points $\Bepm$, $\Aepm$ or $\Eepm$ and
let $\bar{C}_\eps$ be its lift to the universal cover depending continuosly 
on $\eps$.

\begin{lemma}[Monotonicity lemma]\label{lem-monot}

The following assertions hold:

\begin{enumerate}
	\item $\dde(x-y)(\bar{C}_\eps) \to+\infty\text{ as } \eps\to 0^+$, for any choice $\bar{C}_\eps=\Aepmb;\ \Eepmb$

	\item The equation $(y-x)(\Eemnb)= 2\pi n$ has solution  $\eps=\eps_n$ for any~$n$, and $\eps_n=O(1/n)$. 
\end{enumerate}
\end{lemma}
\begin{remark}
	The second assertion of the Lemma implies that for ${\eps=\eps_n}$ the diagonal $$\Delta:=\{y=x\pmod{2\pi\bbZ}\}$$ crosses top left corner of the rectangle~$K_\eps$.
\end{remark}

Lemma~\ref{lem-monot} is proved in Section~\ref{ssec-monot}.

The following lemma describes the graph of the Poincar\'e map near the
points~$\Bepm$.  Denote by $U$ the set of points on $\Gamma$ where the
derivative of the Poincar\'e map is close to 1:

\begin{equation}
	U:=\{x\in S^1 \mid P'_\eps(x) \in [1/2, 2]\}
\end{equation}

\begin{lemma}[Convexity Lemma]
	\label{lem-conv}
	The set $U$ consists of two arcs contained in $\Dep$. On one of them the derivative $P'_\eps$ increases and on the other it decreases.
\end{lemma}

In particular, it follows that there exist exactly two solutions of the equation
$P'_\eps(x)=1$. Lemma~\ref{lem-conv} is proved in the Section~\ref{ssec-conv}.

Now we can prove Main Theorem~\ref{thm-main} modulo these lemmas.

\subsection{Existence of canard solutions}\label{ssec-canard-existence}
In this section we deduce Main Theorem~\ref{thm-main} from 
Lemmas~\ref{lem-shape}, \ref{lem-monot} and \ref{lem-conv}. 

\subsubsection{Heuristic ideas}

We notice that the points of intersection of the graph $\gaeps$ and the diagonal
$\Delta:=\{y=x \pmod {2\pi\bbZ}\}$ correspond to the fixed points of the
Poincar\'e map and therefore to the closed solutions of the system. The stability of
the corresponding cycle depends on the derivative of the Poincar\'e map at the
fixed point: if the derivative is greater than $1$, then the cycle is unstable,
if it is less then $1$, then the cycle is stable. We are particulary interested in
the fixed points that belong to the segment~$\Dep$ because they correspond to
the closed canard solutions. 

Lemmas~\ref{lem-shape} and~\ref{lem-monot} describe  behavior of the system as
${\eps\to 0^+}$. The graph $\gaeps$ lies inside the union of horizontal and
vertical rings, which are exponentially thin when $\eps\to 0^+$. The graph moves
from the lower right corner to the upper left. Theoretically, the following
cases for relative positions of the diagonal $\Delta$ and the graph $\gaeps$ are
possible (Fig.~\ref{fig-poincare-graph}):

1. The diagonal~$\Delta$ does not intersect the rectangle $K_\eps$. In this case
it intersects $\gaeps$ in two points, one of them belongs to the ring $\Pep$,
and corresponds to an unstable fixed point, and the other belongs to the ring
$\Pem$ and corresponds to the stable fixed point from~$\Gamma\setminus\Dep$. 

2. The diagonal~$\Delta$ intersects the rectange~$K_\eps$, but the stable fixed
point is located outside $\Dep$ like in the previous case. (Note that the
position of the unstable fixed point is not significant for our analysis.) 

3. The diagonal~$\Delta$ intersects~$\gaeps$ in two points, and the point of
interstection which corresponds to the stable fixed point belongs to the
rectangle $K_\eps$. In this case this point corresponds to the stable canard
cycle, which provides its existence in Main Theorem,~\ref{thm-main}.  Uniqueness of
the cycle follows from Convexity Lemma~\ref{lem-conv}. 

4. The diagonal~$\Delta$ touches~$\gaeps$ in one of the points~$\Bepm$.

5. The diagonal~$\Delta$ does not intersect~$\gaeps$.

From Lemma~\ref{lem-monot} and the fact that the graph $\gaeps$ depends on
$\eps$ continuosly it follows that when $\eps\to 0^+$, the described cases occur
consecutively and cyclically, in the following order: 1, 2, 3, 4, 5, 4, 3, 2, 1,
\ldots Values of $\eps$ that correspond to case 3 form intervals $C_n^\pm$  mentioned in Main Theorem~\ref{thm-main}.

Later we will prove these propositions strictly.

\begin{figure}
	\psfrag{Dem}{$\Dem$}
	\psfrag{Pepl}{$\Pep$}
	\psfrag{Delta}{$\Delta$}
	\psfrag{Pemn}{$\Pem$}
	\psfrag{Dep}{$\Dep$}
	\psfrag{ps}{$p_s$}
	\psfrag{pu}{$p_u$}
	\psfrag{Epl}{$\Eepl$}
	\psfrag{Emn}{$\Eemn$}
	\psfrag{Aepl}{$\Aepl$}
	\psfrag{Aemn}{$\Aemn$}
	\psfrag{Bpl}{$\Bepl$}
	\psfrag{Bmn}{$\Bemn$}
	\begin{center}
		\begin{tabular}{cc}
			Case $1$ & Case $2$\\
			\includegraphics[scale=0.5]{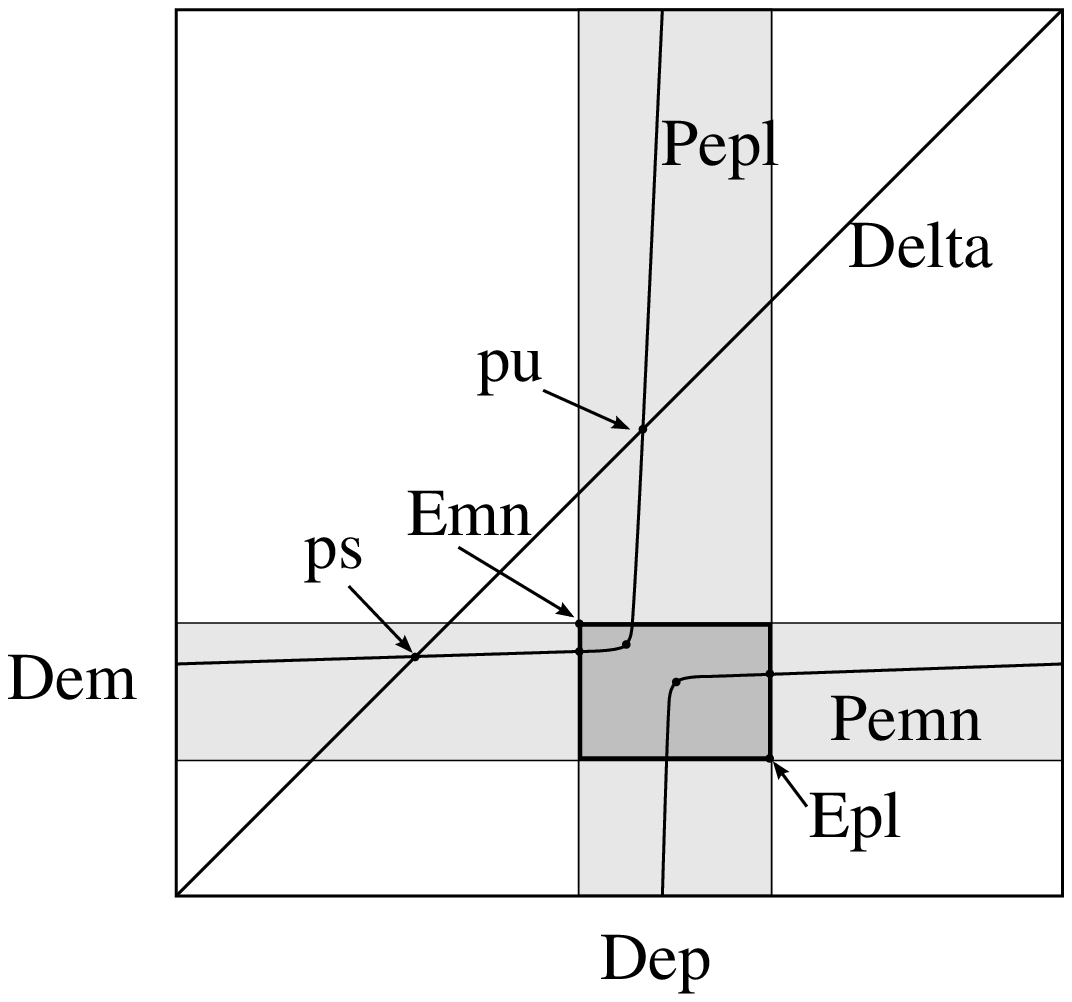}
			&
			\includegraphics[scale=0.5]{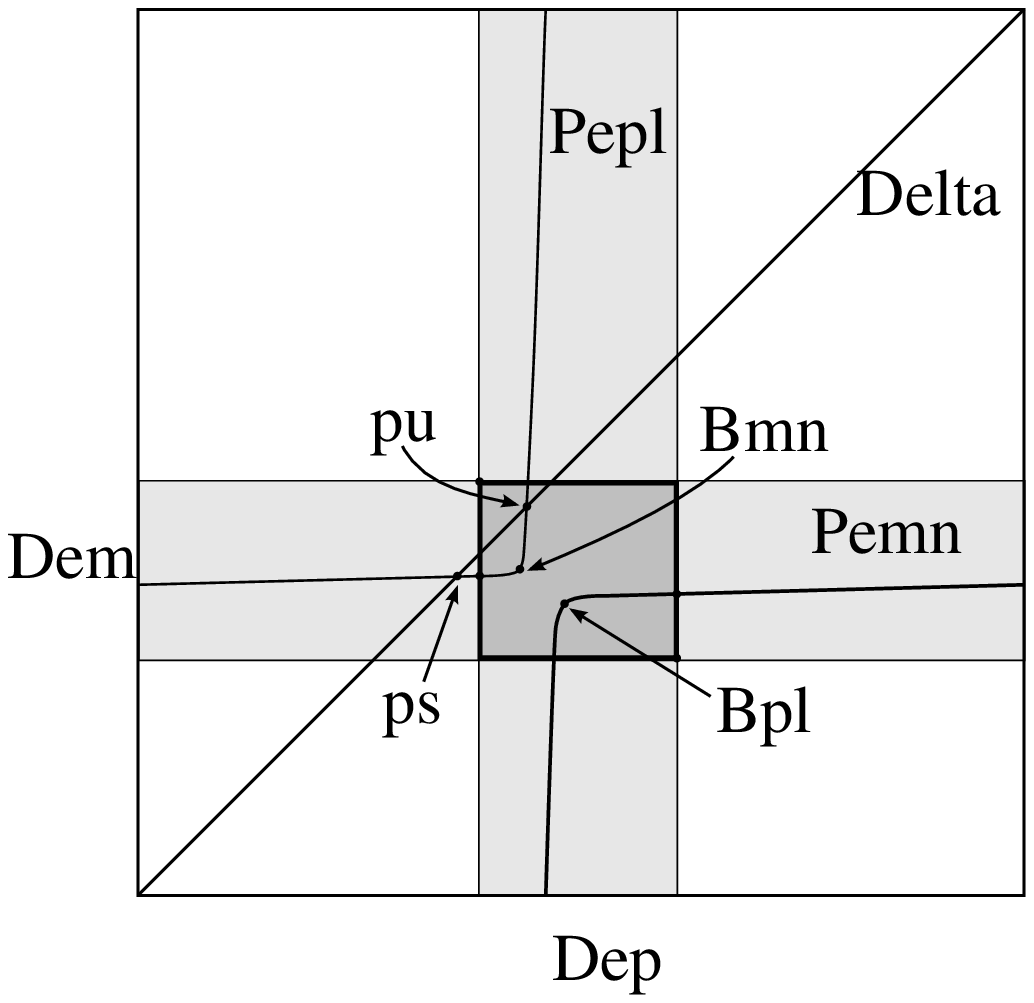}\\
			Case $3$ & Case $4$\\
			\includegraphics[scale=0.5]{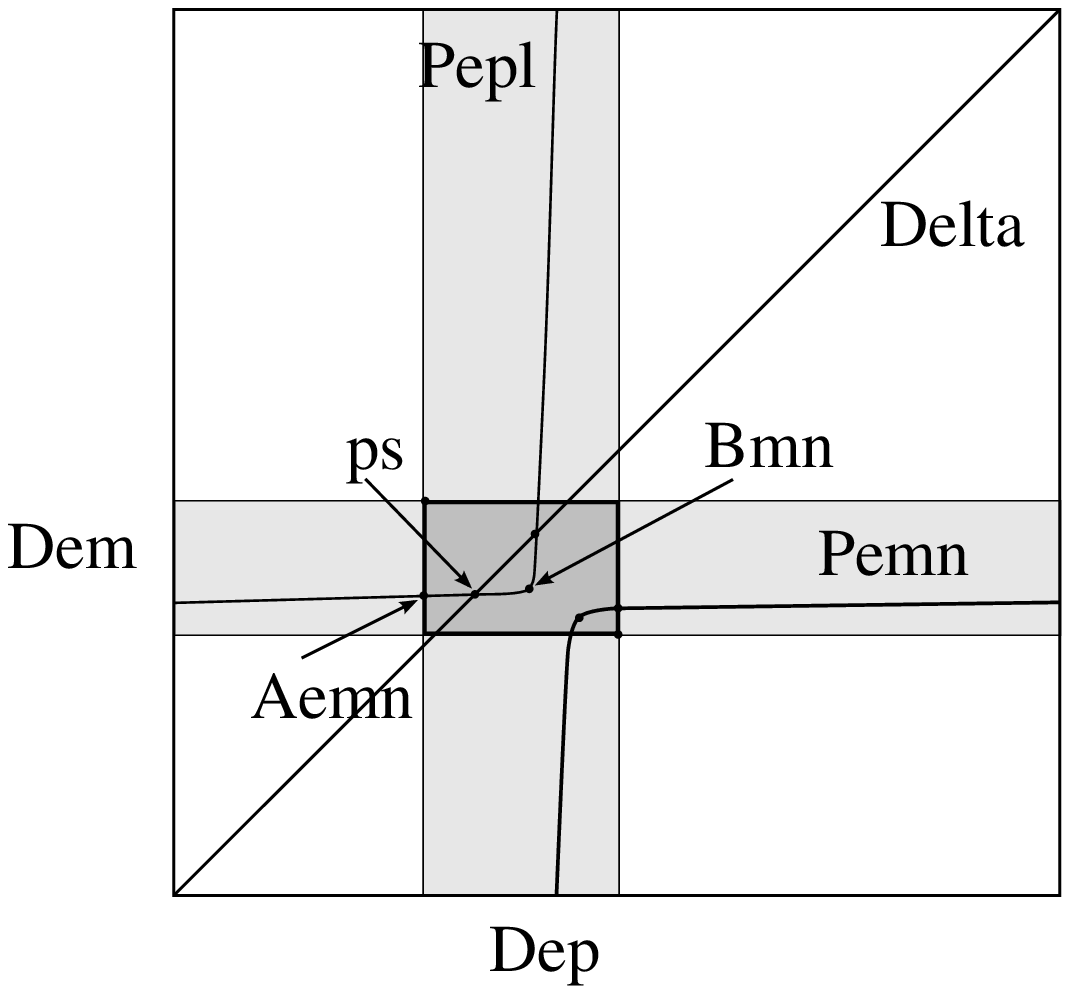}
			&
			\includegraphics[scale=0.5]{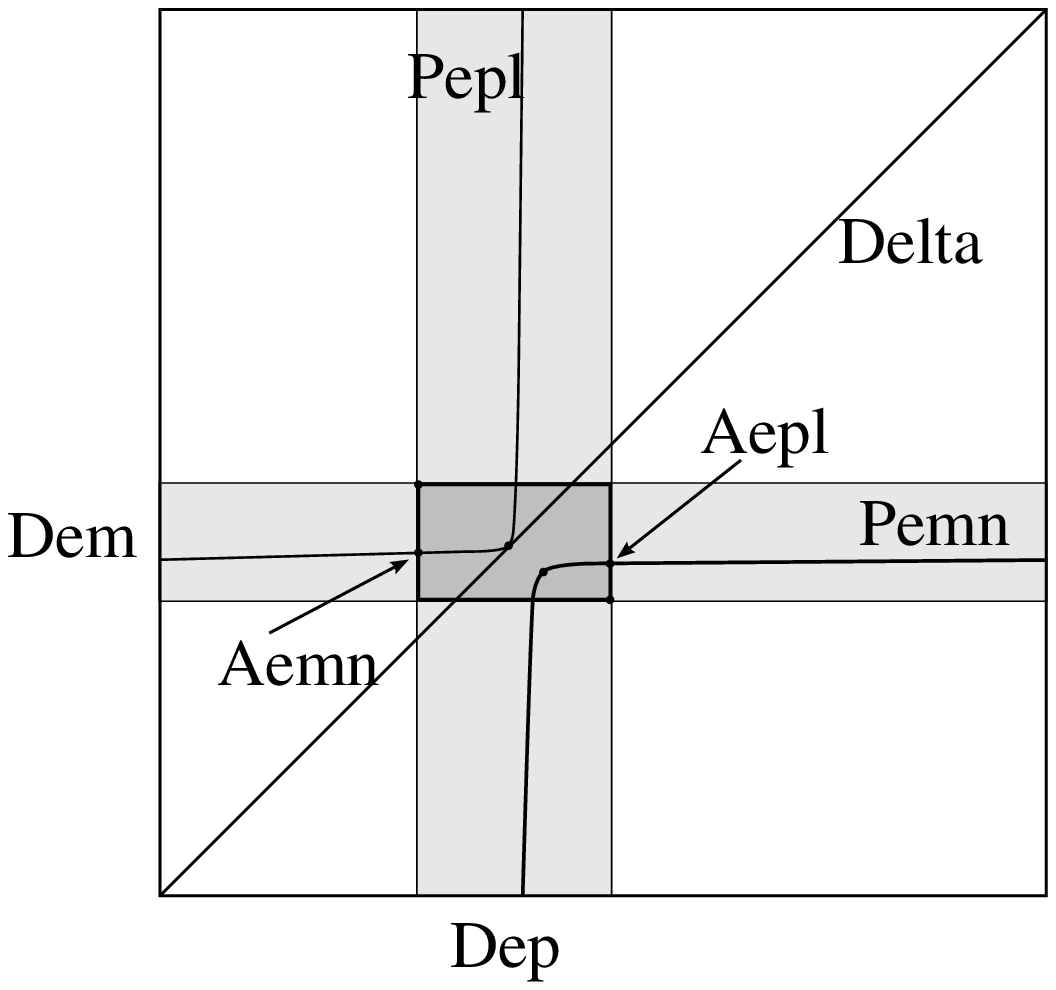}\\
			Case $5$ & Case $4'$\\
			\includegraphics[scale=0.5]{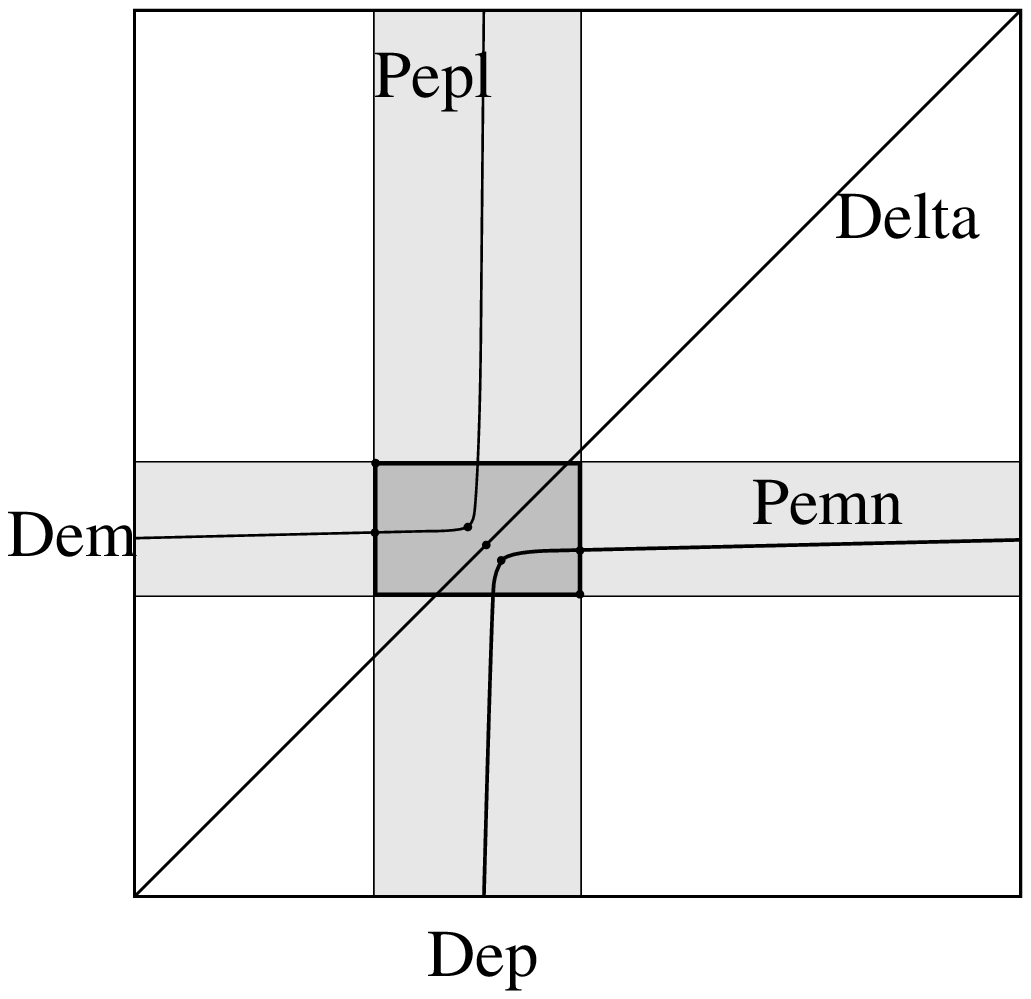}
			&
			\includegraphics[scale=0.5]{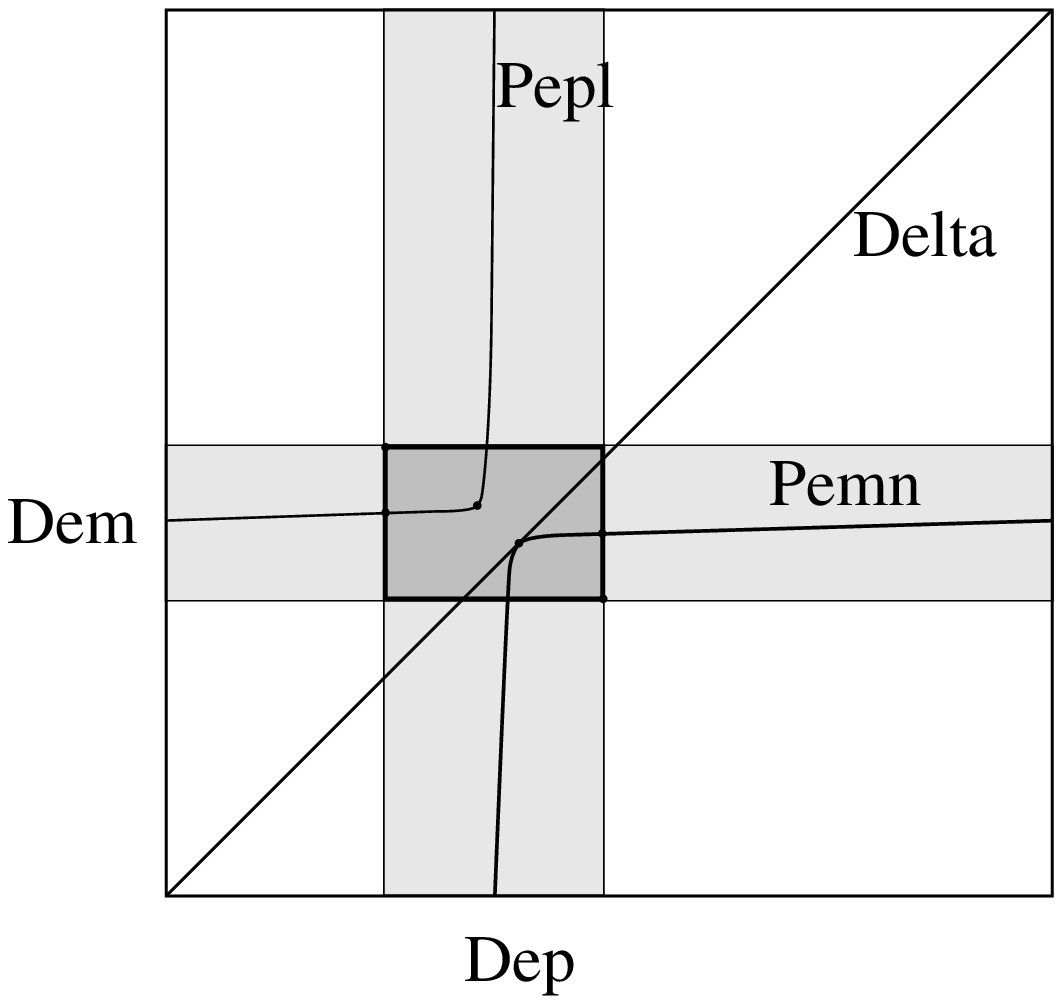}
		\end{tabular}
		\caption{The graph of the Poincar\'e map. As $\eps\to 0^+$ cases 1, 2, 3, 4,
		5 occur consecutively, then case 4 takes place again (denoted by ``case
		$4'$''), then cases 3, 2, 1 (not shown), after which they repeat cyclically.
		}\label{fig-poincare-graph}
	\end{center}
\end{figure}

\subsubsection{Segments on the line of $\eps$}
\begin{proof}[Proof of Main Theorem~\ref{thm-main}]
Define the segments~$R_n$ in the following way: 

\begin{equation}
	R_n:=\{\eps\in(\bbR_+,0) \mid (y-x)(\Eeplb) \le 2 \pi n \le (y-x)(\Eemnb)\}=[\alpha_n,\beta_n]
\end{equation}

In other words, $\eps \in R_n$ when the diagonal~$\Delta$
crosses the rectangle~$K_\eps$. Let us show that in this case~$R_n$
satisfies assertions 1 and 2 of Theorem~\ref{thm-main}.

Assertion 1 of Monotonicity Lemma~\ref{lem-monot} implies that $R_n$ is actually
a segment (as a preimage of a segment under a continuos monotonic map).  Moreover,
$\eps=\alpha_n$ is a solution of the equation $(y-x)(\Eeplb)=2\pi n$ with
respect to $\eps$; according to assertion 2 of Lemma~\ref{lem-monot},
$\alpha_n=O(1/n)$, which proves assertion 2 of the theorem.

Let us estimate the length of the segment~$R_n$, showing
that~$|R_n|=O(e^{-C/\eps})$ for some~$C>0$. Assertion 1 of the Shape
Lemma~\ref{lem-shape} implies, that 

\begin{multline}
|(y-x)(\Eeplb)-(y-x)(\Eemnb)| \le 
 |x(\Eeplb)-x(\Eemnb)|+|y(\Eeplb)-y(\Eemnb)| =\\=
 |\Dep|+|\Dem|=
 O(e^{-C'/\eps})=O(e^{-Cn})
\end{multline}
Together with assertion 1 of the Monotinicity Lemma and the mean value theorem
it implies the required estimate for $|R_n|$.  Hence assertions 1 and 2 of
theorem~\ref{thm-main} are proved for the selected $R_n$.

Let us show that if $\eps$ does not belong to any~$R_n$ then we are in the
domain of case 1 from the previous section. This will imply assertion 3 of 
Theorem~\ref{thm-main}.

\begin{proposition}
	Let~$K_\eps\cap \Delta = \varnothing$. Then the rotation number of $P_\eps$
	is an integer (to be more exact, $\rho(\eps)=0 \pmod {2\pi\bbZ}$)
	and $P_\eps$ has exactly two hyperbolic fixed points (one of them attracting and
	the other one repelling). 
\end{proposition}
\begin{proof}
	The original proof is given in~\cite{GI} (Proposition~1, p.~34).
	We reproduce it here with minor changes.

	On the arc~$S^1\setminus \Dep$, the graph $\gaeps$ of $P_\eps$ has slope less
	than one. The endpoints of this graph lie on the left and the right
	boundaries of the rectangle $K_\eps$ at the points $\Aepl$ and $\Aemn$. If we
	connect these points by a segment inside $K_\eps$, we obtain a closed curve
	$\gamma_h$ on $\bbT^2$.  This curve has homotopy type $(1,0)$, because
	outside $K_\eps$ its slope is less than 1 and it intersects each vertical
	circle $y=\const$ at exactly one point. Consequently, $\gamma_h$ intersects
	$\Delta$ in some point. The Rolle lemma implies that this point is unique,
	because the slope of the curve is less than $1$. Let us denote it as ~$p_s$;
	it cannot lie in $K_\eps$, because $\Delta$ does not intersect~$K_\eps$.
	Therefore, it is a fixed point of~$P_\eps$.
	It is stable because it lies outside~$K_\eps$.
	
	Applying the same arguments to the inverse map $P_\eps^{-1}$ (which is the
	Poincar\'e map after time inversion) we obtain the second (unstable) fixed
	point $p_u$, which also lies outside~$K_\eps$.
\end{proof}

Let us prove assertion 4 of the Main Theorem~\ref{thm-main}: the existence and
uniqueness of canard solutions. Let us define sets $\tCnmn$ and $\tCnpl$ in the
following way: 

\begin{equation}
	\begin{array}{l}
		\tCnmn:=\{\eps\in R_n \mid (y-x)(\Bemnb) < 2 \pi n < (y-x)(\Aemnb)\},\\
		\tCnpl:=\{\eps\in R_n \mid (y-x)(\Aeplb) < 2 \pi n < (y-x)(\Beplb)\}.
	\end{array}
\end{equation}

In other words, there are values of a parameter for which case 3, a stable fixed
point inside $K_\eps$, occurs. The Monotonicity and Convexity Lemmas imply that
for $n$ big enough, the sets $\tCnpm$ are non-empty and open. Therefore, one can
choose two open intervals $\Cnpm\subset\tCnpm$ for each $n$. By definition,
$\Cnmn\cup\Cnpl\subset R_n$.

These intervals do not intersect each other: remark~\ref{rem-B-relations}
implies that $(y-x)(\Bemn)>(y-x)(\Bepl)$ (which means,
$\rarc{(y-x)(\Bemn)}{(y-x)(\Bepl)}\subset\rarc{\Eepl}{\Eemn}$).

Let $\eps \in \Cnmn$ (the case of $\eps \in \Cnpl$ can be treated in a similar
way). Take an arc $\laeps$ of the graph~$\gaeps$ with the endpoints~$\Aemn$ and
$\Bemn$, which lies inside $K_\eps$. By the definition of $\Cnmn$, the endpoints of
$\laeps$ lie on different sides of $\Delta$ in $K_\eps$, and thus it crosses
$\Delta$.  The intersection point belongs to~$K_\eps$. Convexity
Lemma~\ref{lem-conv} implies that this point is unique and the slope of $\gaeps$
at this point is less than $1$, so it corresponds to the stable fixed point of
the Poincar\'e map.

Poincar\'e map is a well-defined diffeomorphism of a circle, so the existence of
a stable fixed point implies the existence of an unstable fixed point.

Thus, theorem~\ref{thm-main} is proved.
\end{proof}

\section{Normalization}\label{sec-norm}
In this and the next two sections, we prove Lemmas~\ref{lem-shape}, \ref{lem-monot}
and~\ref{lem-conv}. Lemma ~\ref{lem-shape} is proved in
subsection~\ref{ssec-shape}.  Lemma~\ref{lem-monot} is proved in subsection
~\ref{ssec-monot}.  Lemma~\ref{lem-conv} is proved in subsection~\ref{ssec-conv}
modulo an auxiliary Theorem~\ref{thm-poincare-jump-est}. That theorem is
proved in section~\ref{sec-technique}.

We need two theorems (due to Guckenheimer and Ilyashenko) that describe normal
forms of slow-fast systems on the two-torus.

\subsection{Nonlinear transition}\label{ssec-nonlin-trans}
The following theorem describes slow-fast dynamics on the two-torus outside of
some neighborhood of a slow curve.

\begin{theorem}[\cite{GI}]
	\label{thm-nonlin-trans}
	Consider a vector field on the cylinder $x\in S^1=\bbR/\bbZ$, $y\in \bbR$ defined by 

	\begin{equation}\label{eq-cyl-sf}
		\dtx = f(x,y,\eps),\quad \dty=\eps g(x,y,\eps),\qquad  f>0,\ g>0
	\end{equation}
	Let $a,b \in \bbR,\ a<b$ and $\eps>0$. The Poincar\'e map $\rPeab{a}{b}$ of a cross-section $\Gamma^a = \{y=a\}$ to $\Gamma^b = \{y=b\}$ has the form 

	\begin{equation}
		\rPeab{a}{b}(x)=G_\eps^1\circ (G_\eps^2(x)+T(\eps)),
	\end{equation}
	where $T(\eps)\to \infty$ as $\eps\to O^+$ and $G^{1,2}_\eps$ are
	diffeomorphisms of a circle. $G^{1,2}_\eps \to G_{1,2}$ as $\eps\to 0^+$ and
	$G_{1,2}$ are diffeomorphisms of a circle, as well. 
\end{theorem}
In other words, the Poincar\'e map from one fixed vertical cross-section to
another fixed vertical cross-section is a rotation by $T(\eps)$ in properly
selected coordinates on the preimage and image circles. $T(\eps)$ tends to
infinity as $\eps\to 0^+$. It is important that corresponding coordinate maps
have uniformly bounded derivatives as $\eps\to 0^+$ and tend to smooth maps. 

Theorem~\ref{thm-nonlin-trans} is proved in~\cite{GI} (Theorem 2, p.~35).

\subsection{Normalization near the slow curve}
The following theorem shows that near the slow curve and outside of some
neighborhood of the jump point system~\bref{eq-main} is smoothly equivalent to
a linear system. 

\begin{theorem}[\cite{GI}]\label{thm-normnearss}
Consider the system

\begin{equation}
	\dtx=f(x,y,\eps),\quad \dty=\eps g(x,y,\eps),\qquad g>0
\end{equation}
Let the corresponding fast system have a curve of nondegenerate fixed points (which
is the nondegenerate slow curve). Then for small $\eps>0$ near this curve (and
outside any small fixed neighborhood of jump points) the system is smoothly orbitally
equivalent to the family 

\begin{equation}\label{eq-normform-nearss}
\dtx = a(y,\eps)x,\quad \dty=\eps
\end{equation}
\end{theorem}
The proof of this theorem can be found in~\cite{GI} (Theorem 3, p.~38). The Fenichel theorem (see ~\cite{Fenichel}, and also~\cite{ODA1} and ~\cite{ODA2})
implies that in some neighborhood of the stable (unstable) parts of the slow
curve $M^-$ ($M^+$) there exists a smooth invariant manifold $\Semn$ ($\Sepl$),
which can be seen as the graph of the function $x=s^-(y,\eps)$
($x=s^+(y,\eps)$ resp.). Variational equations imply that the function $a(y,\eps)$ in Theorem~\ref{thm-normnearss} has the form:

\begin{equation}\label{eq-afx}
a(y,\eps)=f'_x (s(y,\eps),y,\eps)
\end{equation}
\begin{remark}
	True slow curves~$S^\pm_\eps$ are non-unique. However, the distance between any such curves is of order~$O(e^{-C/\eps})$ (thus they are exponentially close to each other) and we can choose any of them for our analysis.
\end{remark}
\subsection{Rough estimate of the derivative of the Poincar\'e map}
Let us prove a preliminary estimate for the derivative of the Poincar\'e map.

\begin{remark}
	Here and below we will use letter $C$ (without indicies) to denote (generally
	different) positive constants that do not depend on~$\eps$. 
\end{remark}
\begin{lemma}\label{lem-robust}
Consider the slow-fast system

\begin{equation}
	\dtx=f(x,y,\eps),\quad\dty=\eps,\qquad (x,y)\in S^1\times \bbR
\end{equation}
For $\eps>0$ small enough the following estimate for Poincar\'e map $\rPeab{a}{b}:
\Gamma^a \to \Gamma^b$ holds: 

\begin{equation}
	|\ln (\rPeab{a}{b}(x))'| < C \frac{b-a}{\eps},
\end{equation}
\end{lemma}
\begin{proof}

Let $x=x(y; x_0,\eps)$ be the phase curve that contains the point $(x_0,a)$ for
given $\eps$. As $f'_x$ is bounded (from above and from below) on the
two-torus, variational equations imply that
\begin{equation}
	(\rPeab{a}{b}(x_0))'=\exp\left(\frac{1}{\eps}\int_{a}^{b}
	f'_x(x(y;x_0,\eps),y,\eps) dy\right) <e^{\frac{C}{\eps}(b-a)}
\end{equation}
The same estimate holds for the inverse map. Applying the logarithm we obtain the
required estimate.
\end{proof}

This lemma will be used in the neighborhood of a jump point. More precise
estimates of the derivative of the Poincar\'e map near a jump point are presented
in subsection~\ref{ssec-jump-est-proof}.

\section{Properties of the Poincar\'e map}\label{sec-poincare}

In this section and two sections that follow we prove Lemmas~\ref{lem-shape}, ~\ref{lem-monot} and~\ref{lem-conv}.

\subsection{Distortion: proof of Lemma~\ref{lem-shape}}\label{ssec-shape}
Recall the notation:
\begin{gather*}
	\alpm =\tpm\pm \delpm\\
	\Gapm=\{y=\alpm\}\\
	\Jpl:=\{(x,y) \in \Gapl \mid x \in [\spl,\pi]\},\quad \Jmn:=\{(x,y) \in \Gamn
	\mid x \in [-\pi, \smn]\} 
\end{gather*}

Denote $\rPeab{-\pi}{\alpl}$ as $Q_\eps$ for brevity. Recall that  $\Dep=\inv{Q_\eps}(\Jpl)$.

\begin{proposition}\label{prop-q-est} 
	For any~$\delpl$ small enough there exist constants $c_{1,2}(\delpl)>0$ such that the following estimates hold:

	1. $|\Dep|=O(e^{-c_1(\delpl)/\eps})$; 

	2. $Q_\eps'\le C e^{c_2(\delpl)/\eps}$.

	Moreover, $c_{1,2}(\delpl)\to 0^+$ as $\delpl\to 0^+$.
\end{proposition}
\begin{remark}\label{rem-rev-q-est}
	Let us denote the Poincar\'e map from $\Gamn$ to $\Gamma$ (in forward time) as
	$\wtQ_\eps$. For the system with the reversed time similar assertions can be made
	for $\wtQ_\eps$ and some other constants $c_{1,2}$: 
	
	1. $|\Dem|=O(e^{-c_1(\delmn)/\eps})$; 
		
	2. $\left(\wtQ_\eps^{-1}\right)'\le C e^{c_2(\delmn)/\eps}$. 

	Moreover, $c_{1,2}(\delmn)\to 0^+$ as $\delmn\to 0^+$.

	It can be proven by the same arguments, applied to the system with
	the reversed time. 
\end{remark}
	
\begin{proof}[Proof of proposition~\ref{prop-q-est}]
	In order to prove the first assertion let us estimate the derivative of~$Q_\eps$ on the segment $\Jpl$. Consider cross-sections

	\begin{equation}
		\Ga_1=\{y=\tpl+\delpl^2\},\quad \Ga_2=\{y=\tpl-\delpl^2\}.
	\end{equation}
	Denote a segment $[\tau^++\delpl^2, \tau^--\delmn^2]$ by $\Sigma$ and let
	$\Semn$ and $\Sepl$ be the stable and unstable parts of the true slow curve
	over the segment $\Sigma$, resp. Also, denote  the segment $[\tau^++\delpl^2,
	\tau^++\delpl]$ by $\Sigma^+$.

	The map $\Qeinv$ can be presented as a composition of the following Poincar\'e
	maps:

	\begin{gather*}
		\Qeinv: \Ga^+\stackrel{Q_1}{\to} \Ga_1 \stackrel{Q_2}{\to} \Ga_2 \stackrel{Q_3}{\to} \Ga\\
		\Qeinv=Q_3\circ Q_2 \circ Q_1
	\end{gather*}
	We will use normal forms to analyse maps $Q_1$ and $Q_3$ and Lemma~\ref{lem-robust} to estimate the derivative of~$Q_2$. 

	\subsubsection{Linear contraction: neighborhood of the true slow curve}
	Let us prove the following estimate:
	\begin{equation}
		\label{eq-linear-contraction}
		\left.Q_1'\right|_{\Jpl}\le C_1 \exp \left(-\frac{C\delpl^{3/2}}{\eps}\right),\quad
	\end{equation}
	where $C, C_1$ are positive constants that do not depend on ~$\delpl$.
	
	When trajectories that cross $\Jpl$ are traced back in time from $\Gapl$ to
	$\Ga^1$, we see that the amount of time they spend outside of some
	neighborhood of~$M^-$ is uniformly bounded from above. Thus we can use the
	normal form~\bref{eq-normform-nearss} to study the map $Q_1$ (see
	Theorem~\ref{thm-normnearss}). 

	We shift coordinates by moving the origin to the jump point $\Gpl=(\tpl,\spl)$:

	\begin{equation}
		x_1=x-\spl,\quad
		y_1=y-\tpl
	\end{equation}
	We will apply normal form~\bref{eq-normform-nearss} over the segment
	$\Sigma^+$ for small ~$\delpm$. (Note that $\Sigma^+$ does not depend
	on~$\eps$.) Let us show that for $y_1\in\Sigma^+$ the function $a(y_1,\eps)$
	has the following form: 

\begin{equation}
	a(y_1,\eps)=(C+o(1))\sqrt{y_1}+O_{\delpl}(\eps),
\end{equation}
where $o(1)$ is a function of $y_1$ and $\eps$ which tends to $0$ uniformly as
$\delpl\to 0$ and $O_{\delpl}(\eps)$ is of order $\eps$ for any fixed~$\delpl$.
Indeed, from nondegenericity assumptions~\bref{eq-nondeg-main-wlog} and the
implicit function theorem, it follows that~$M^+$ can be presented as a graph of
the function

\begin{equation}
	x=s^+(y_1)=(C+o(1))\sqrt{y_1}.
\end{equation}
Thus the true slow curve is a graph of the function

\begin{equation}\label{eq-tsc-asy}
	x=s^+(y_1,\eps)=(C+o(1))\sqrt{y_1}+O_{\delpl}(\eps).
\end{equation}
Substitute~\bref{eq-tsc-asy} into the expression for $a(y_1,\eps)$
(see~\bref{eq-afx}). We obtain:

\begin{multline}\label{eq-a-asy}
a(y_1,\eps)=f'_x(s(y_1,\eps),y,\eps)=\\
=f''_{xx}(0,0,0) (s(y_1)+O_{\delp}(\eps)) + O(|y|+\eps)=
\\= (C+o(1))\sqrt{y_1}+O_{\delp} (\eps),
\end{multline}
where $f''_{xx}|_0>0$ according to~~\bref{eq-nondeg-main-wlog}. 

From the normal form~\bref{eq-normform-nearss} it follows that the map $Q_1$
is linear in normalized chart. To simplify the notation, let $\delta=\delpl$.
Substituting expression~\bref{eq-a-asy} for $a(y,\eps)$ into ~\bref{eq-normform-nearss} and integrating, we have:

\begin{equation}
	Q_1(x_1)=\Lambda(\eps) x_1,\quad \Lambda(\eps)=e^{-\frac{\delta^{3/2}}{\eps}(C+o(1)+O_{\delta}(\eps))}<e^{-\frac{\delta^{3/2}}{\eps}(\frac{3}{4}C)}
\end{equation}
The normalizing map that we applied to obtain the normal form has its derivative
bounded away from $\infty$ and $0$. Thus in the original coordinates the
right-hand part of~\bref{eq-linear-contraction} can be taken to be a constant
times $\Lambda(\eps)$.

\subsubsection{Neighborhood of the jump point}

Let us show that for~$\delta$ small enough the expansion of the phase curves,
which is possible near the jump point (during the transition from $\Ga_1$ to
$\Ga_2$) is much less than their contraction that was accumulated during the
transition near the slow curve. Indeed, Lemma~\ref{lem-robust} yields the
following estimate: 

\begin{equation}
	Q_2' \le e^{C \delta^2/\eps}\ll \Lambda^{-1}(\eps).
\end{equation}

\subsubsection{Estimating the derivative of~$\Qeps$}
Map $Q_3$ is a rotation in rectifying charts. Its derivative is bounded
uniformly in $\eps$.

Using the chain rule, we get:

\begin{equation}\label{eq-Qeinv-est}
	\Qeinv|_{\Jpl}'=O(e^{-C/\eps}), \quad C=O(\delpl^{3/2})
\end{equation}
Thus,

\begin{equation}
	|\Depl|=O(e^{-C/\eps}),\quad C=O(\delpl^{3/2})
\end{equation}
Assertion 1 of Proposition~\ref{prop-q-est} is proved. Due to
Remark~\ref{rem-rev-q-est} this proves assertion 1 of Lemma~\ref{lem-shape}. 

Let us prove assertion 2 of Proposition~\ref{prop-q-est}. Consider the map $Q_\eps$:

\begin{equation}
	Q_\eps=\inv{Q_1}\circ \inv{Q_2}\circ \inv{Q_3}
\end{equation}
The derivative of~$\inv{Q_3}$ can be estimated from above by
Theorem~\ref{thm-nonlin-trans} and so is bounded by a constant. The time of the
transition from $\Ga_2$ to $\Gapl$ is no greater than $2\delpl$.  Applying
Lemma~\ref{lem-robust} to the map~$\inv{Q_1}\circ \inv{Q_2}$, we obtain
the desired estimate for its derivative. Along with the chain rule, it completes
the proof. 
\end{proof}

Let us prove assertion 4 of Lemma~\ref{lem-shape}. We present $P_\eps$ as
a composition: 

\begin{equation}
	P_\eps: \Ga\stackrel{Q_\eps}{\to} \Gapl \stackrel{\Ppm}{\to} \Gamn
	\stackrel{\wtQ_\eps}{\to} \Ga
\end{equation}
By definition, the trajectories that start outside of $\Dep$, intersect $\Gapl$
outside of segment $\Jpl$:

\begin{equation}
	Q_\eps(\Gamma\setminus\Dep)= \Gapl\setminus J^+
\end{equation}
Note that outside of any neighborhood of the slow curve the function $f$ is
bounded away from zero. Therefore it will take time of order $O(1)$ for the
trajectory that crosses  $\Gapl\setminus\Jpl$ to reach certain neighborhood of
the stable part of the slow curve $M^-$. Due to the normal
form~\bref{eq-normform-nearss}, the derivative of the map~$\Ppm$ satisfies the
following estimate:

\begin{equation}\label{eq-ppm-est}
	(\Ppm)'|_{\Gapl\setminus\Jpl}<  e^{-C/\eps},
\end{equation}
where~$C$ is bounded away from $0$ as $\delpm\to0$. 

Therefore, the image of segment~$\Gapl\setminus\Jpl$ has exponentially small
length and intersects the true slow curve. Thus $\Ppm(\Gapl\setminus\Jpl) \subset
\Jmn$ for $\eps>0$ small enough. We have: 
\begin{equation}
	P_\eps(\Ga\setminus\Dep)\subset \Dem,
\end{equation}
which implies assertion 4 of Lemma~\ref{lem-shape}.

Let us prove assertion 2 of Lemma~\ref{lem-shape} (assertion 3 can be proven by
the same arguments applied to the system with reversed time). From
estimate~\bref{eq-ppm-est}, it follows that for a trajectory that crosses
$\Ga\setminus\Dep$, the derivative of the map $\Ppm$ is less than $1$ and,
moreover, it is exponentially small as $\eps\to0^+$. Similar
to~\bref{eq-Qeinv-est}, one can show that the derivative of~$\wtQ_\eps$ is
exponentially small on the segment~$\Jmn$.  Derivative of the map~$Q_\eps$ is
less than $O(e^{c(\delpl)/\eps})$ according to assertion 2 of
Proposition~\ref{prop-q-est} and $c(\delpl)$ can be considered small enough for
an appropriate choice of~$\delpl$.  Therefore, any possible expansion
by~$Q_\eps$ is controlled by the exponential contraction $\Ppm$, and the
derivative of $P_\eps$ is exponentially small.

Lemma~\ref{lem-shape} is proven.

\subsection{Convexity: proof of Lemma~\ref{lem-conv}}\label{ssec-conv}
In this section Lemma,~\ref{lem-conv} will be proved.
Recall that we have defined the set $U$ of initial conditions $u\in \Gamma$, for
which

\begin{equation}\label{eq-u-cond}
	P_\eps'(u)\in[1/2,2].
\end{equation}
We will show that $U$ consists of two arcs of the circle~$\Gamma$: on one arc
the derviative of the Poincar\'e map increases (the graph $\gaeps$ is convex), and on the
other one it decreases. All along the proof we will assume that the initial
condition of every trajectory considered belongs to $U$. 

\subsubsection{Heuristic motivation}
As Lemma~\ref{lem-shape} shows, outside of the segment~$\Dep$ the derivative of
the Poincar\'e map $P_\eps$ is exponentially small. Therefore, the set~$U$  lies
in $\Dep$ and corresponding trajectories cross $\Jpl$.

After intersecting~$\Jpl$ any such trajectory spends some time near the unstable
part of the slow curve~$M^+$. Then it jumps either up or down, leaves the
neighborhood of $M^+$, and approaches the neighborhood of the stable part of
true slow curve~$M^-$ in bounded time. Calculating the derivative of the
Poincar\'e map, one easily sees that it increases while the trajectory spends
time near~$M^+$ and decreases while the trajectory passes along $M^-$ (which
follows from the normal form~\bref{eq-normform-nearss}). We do not consider
trajectories which spend too much time near either~$M^+$ or $M^-$: the
derivative of the Poincar\'e map for such trajectories is either too big or too
small. Trajectories we are interseted in jump ``somewhere in between'' and spend
comparable amount of time near $M^+$ and $M^-$.  (Later we will give rigorous
definitions for this.)

After the jump (either upwards or downwards), those trajectories cross~$J^-$, jump near the point~$\Gmn$
and cross~$\Gamma$ in a point that belongs to $\Dem$.

Let us extend the true slow curve $M^+_\eps$ in reverse time and denote the
point of its intersection with the cross-section~$\Gamma$ by $u_0$. Obviously,
$u_0\in\Dep$. Consider trajectories starting from points $u\in \Dep$ that lie
lower than~$u_0$. When $u$ tends to~$u_0$ from below, the corresponding
trajectory tends to $\Sepl$ and thus it spends more time near~$M^+$ and less
time near~$M^-$.  Hence when $u$ increases, the derivative of the Poincar\'e map
\emph{increases} too.  When $u$ coincides with $u_0$, the derivative reaches its
maximum value, because the corresponding trajectory coincides with the true slow
curve~$\Sepl$.  After that, as u increases, the derivative of the Poincar\'e map will
decrease for similar reasons: the trajectory spends less time near~$M^+$, and
more time near~$M^-$.

The foregoing analysis shows that these naive arguments do work.

\subsubsection{The strategy}

We will use the same method to deal with~$P_\eps$ as we used to prove
Lemma~\ref{lem-shape}: we decompose the map~$P_\eps$ into the composition of
several Poincar\'e maps. By analyzing the dynamics near~$M^+$ and $M^-$, we will
show that the trajectories for which the derivative of
the Poincar\'e map is close to~$1$ leave the neighborhood of~$M^+$ near some fixed
cross-section $y=\yptom$. Afterwards, using variational equations, we will estimate
the second derivative of the Poincar\'e map. 

To proceed with this strategy we need additional information describing the 
dynamics near the jump point.

\begin{theorem}
	\label{thm-poincare-jump-est}
	For some constant $\la>0$ and arbitrarily small~$\delpl$, there exists
	a~$C^+(\delpl)$, such that, for arbitrary~$x\in\Dep$, the following
	representation of the derivative of the Poincar\'e map
	$R_\eps=Q_\eps^{-1}=\lPeab{\alpl}{-\pi}$ holds:

	\begin{equation}
		\ln R_\eps'(x)=\frac{C^+(\delpl)+O(\eps^\lambda)}{\eps},
	\end{equation}
	where~$C^+(\delpl)<0$ is continuous and tends monotonically to zero
	as~$\delpl\to0^+$

\end{theorem}
\begin{remark}
	A similar expression (with another constant $C^-(\delmn)<0$, which tends to 0
	monotonically as $\delmn\to0$) holds for $\wtQ_\eps=\rPeab{\almn}{\pi}$.
\end{remark}
Theorem~\ref{thm-poincare-jump-est} is proved in section~\ref{sec-technique}.

We denote the normalizing charts near unstable and stable parts of the true slow
curve by $(\xpl,y)$ and $(\xmn,y)$, resp. Define $\Upm:=\{|\xpm|<b\}$ for some
$b>0$.  Let us first deal with the case when the trajectory jumps from
$S^{+}$ to $S^{-}$ in the negative direction (down). Suppose that the
trajectory under consideration leaves the neighborhood~$\Upl$ when $y=\ypl$ and
reaches~$\Umn$ when~$y=\ymn$.  Obviosly, $\ymn=\ypl+O(\eps)$ (since outside of
$\Upm$, the function $f$ is bounded away from $0$). We assume that in the charts
$(\xpm,y)$, system (\ref{eq-main}) has the following form: 

\begin{equation} 
	\dtx_{\pm} = \apm(y,\eps) \xpm,\quad \dty = \eps 
\end{equation}
Let us define

\begin{equation}
	\Phpm(y,\eps)=\int_{\alpm}^{y} \apm(v,\eps)\,dv, \quad \Phpm(y)=\Phpm(y,0)
\end{equation}
Function~$\Phmn$ (\emph{resp.} $\Phpl$) is equal to the logarithmic derivative of the
corresponding Poincar\'e map in the normalizing chart. In other words, they
estimate contraction (expansion) of trajectories, accumulated during the transition
near stable (\emph{resp.} unstable) part of the slow curve. Since $\apl(y,\eps)>0$ and
$\amn(y,\eps)<0$, and $\alpl<y<\almn$, it follows that $\Phpm(y,\eps)>0$,
and~$\Phpl$ increases with $y$, while $\Phmn$ decreases.  Let $y=\yptom$ be the
root of the equation: 

\begin{equation}
	\Phpl(y,0) - \Phmn(y,0)=0. 
\end{equation}
In order to satisfy~\bref{eq-u-cond} it is neccessary that $\ypm$ is close to
$\yptom$. Otherwise, either attraction or repulson will dominate in~$P_\eps'$,
though~\bref{eq-u-cond} demands them to annihilate. The next Lemma formalizes
this heuristic arguments. 

\begin{lemma}\label{lem-yptom-est} 
	There exists $\la>0$, such that for an appropriate choice of a small~$\delpm$
	and for any trajectory with initial condition in~$U$, the following estimates
	hold:

\begin{equation}\label{eq-yptom-est}
|\yptom-\ypl| = O(\eps^{\lambda}),\quad
|\yptom-\ymn| = O(\eps^{\lambda}).\quad
\end{equation}
	These two estimates are equivalent because $|y^+-y^-|=O(\eps)$.
\end{lemma}
\begin{proof}
Let us decompose the Poincar\'e map $P_\eps$:

\begin{gather*}
	P_\eps: \Gamma\stackrel{Q_\eps}{\to} \Gapl \stackrel{\Ppm}{\to} \Gamn
	\stackrel{\wtQ_\eps}{\to}\Gamma\\
	P_\eps=\wtQ_\eps\circ\Ppm\circ Q_\eps,
\end{gather*}
where $Q_\eps=R_\eps^{-1}$. The chain rule implies: 

\begin{equation}\label{eq-P-eps-prim-decomp}
	\ln P_\eps'=\ln \wtQ_\eps'\circ \Ppm \circ Q_\eps+\ln (\Ppm)'\circ Q_\eps+\ln Q_{\eps}'
\end{equation}
First, we estimate the second term of this sum from below.
Theorem~\ref{thm-normnearss} and variational equations imply: 

\begin{multline}
|\ln (P_{+}^{-})'|  = \left| \frac{1}{\eps} (\Phpl(\ypl,\eps)-\Phmn(\ymn,\eps))\right|=\\
=\left| \frac{1}{\eps} (\Phpl(\ypl)-\Phmn(\ymn))\right|+O(1)=\\
= \frac{1}{\eps}|\Phpl(\ypl)-\Phpl(\yptom)-\Phmn(\ymn)+\Phmn(\yptom)|+O(1) = \\
= \frac{1}{\eps} \left|-\int_{\ypl}^{\yptom} \apl(y,0) dy + \int_{\ymn}^{\yptom} \amn(y,0) dy \right| +O(1) >{} \\
{}> c\frac{|\ypl-\yptom|}{\eps},
\end{multline}
where $c=\frac{1}{2} \min_{\Sigma} (\apl(y,0)-\amn(y,0))>0$ since $\apl(y,0)>0$
and $\amn(y,\eps)<0$. 

According to theorem~\ref{thm-poincare-jump-est},

\begin{equation}
	\ln Q'_\eps = \frac{C_+(\delpl) + O(\eps^{\lambda})}{\eps}, \quad \ln\wtQ_\eps' = \frac{-C_-(\delmn)+O(\eps^{\lambda})}{\eps}
\end{equation}
It follows from the asymptotic behaviour of~$C_\pm$ that one can find
small~$\delpm$ such that $C_+(\delpl)=C_-(\delmn)$.  For such $\delpm$,
condition~\bref{eq-u-cond} implies:

\begin{equation}
	\ln 2\ge |\ln Q'_\eps + \ln \Ppm + \ln \wtQ_\eps'| >
	c\frac{|\ypl-\yptom|+O(\eps^\lambda)}{\eps}.
\end{equation}
Therefore,
\begin{gather}
	|\ypl-\yptom| = O(\eps) + O(\eps^\lambda)= O(\eps^\lambda),\quad \lambda>0\\
	\yptom=\ypl+\eps^\lambda k_1(\ypl,\eps),
\end{gather}
where $k_1(\ypl,\eps)$ is a smooth function.

\end{proof}

Let us take the derivative of~\bref{eq-P-eps-prim-decomp}:
\begin{equation}\label{eq-log-deriv-decomp}
	\dd{}{u}\log P'_\eps = \dd{}{u} \log \wtQ'_\eps \circ P_+^- \circ Q_\eps (u) + \dd{}{u} \log(P_+^-)' \circ Q_\eps(u) + \dd{}{u}\log Q'_\eps(u)
\end{equation}
We will show shat the sign of the lograthmic derivative of the Poincar\'e map
depends only on the sign of the second term in this expression. The other two
terms's influence can be estimated from above:

\begin{gather*}
\left|\dd{}{u}\log Q'_\eps(u)\right| < C\cdot \exp{\left(\frac{\delpl^{3/2}}{\eps}\right)}\\
\left|\dd{}{u}\log \wtQ'_\eps(u)\right| < C\cdot \exp{\left(\frac{\delmn^{3/2}}{\eps}\right)}
\end{gather*}
Proof of these estimates can be found in~\cite{GI}, p. 44. It only requires the
fact that~$f$ and its derivative are bounded and that assertion 2 of 
Proposition~\ref{prop-q-est} holds. Once it is assured, the proof
from~\cite{GI} works without any changes. 

\begin{proposition}\label{prop-Ppm-log-deriv-est}
Suppose $\eps$ and $u$ such that~\bref{eq-u-cond} holds. Let

\begin{equation}
	I=2 \Phpl(\yptom)=\Phpl(\yptom)+\Phmn(\yptom).
\end{equation}
Then assuming that~$\eps$ is sufficiently small, we have the following estimate: 

\begin{equation}\label{eq-mid-term-estimate}
	\left|\dd{}{u} \ln (P_{+}^{-})' \circ Q_\eps(u)\right| > \exp \frac{I}{5 \eps}
\end{equation}
\end{proposition}

It is easy to see that provided~\eqref{eq-mid-term-estimate} the second term of
the sum~\bref{eq-log-deriv-decomp} dominates. Thus, when
trajectories jump down, the whole
expression~\bref{eq-log-deriv-decomp} is positive. Similar arguments show that the derivative is
negative if trajectories jump up. The rest of the section is devoted to
the proof of Proposition~\ref{prop-Ppm-log-deriv-est}, which thus concludes the
proof of Lemma~\ref{lem-conv}. 

\begin{proof}
Let us consider the normalizing chart $\xi$ on the cross-section~$\Jpl$  near
the slow curve, and the normalizing chart  $\eta$ on the cross section~$\Jmn$.
For the trajectory that intersects $\Jpl$ in~$\xi$, we denote its intersection
with~$\Jmn$ by~$\eta(\xi)$.  In the case we are considering (when trajectories
jump down) $\xi$ is negative. The function $\xi \mapsto \eta(\xi)$ defines the
Poincar\'e map $P_{+}^{-}$ in normalizing charts. As direct calculations show
(see \cite{GI}, p. 43), we have:

\begin{equation}\label{eq-eta-of-xi}
	\eta(\xi) = -\xi \exp\left( \frac{\Phpl(\ypl,\eps) - \Phmn(\ymn,\eps)}{\eps}\right)
\end{equation}
However, according to~\bref{eq-yptom-est}:

\begin{multline}\label{eq-Phpl-Phmn}
\Phpl(\ypl,\eps)-\Phmn(\ymn,\eps) - (2 \Phpl(\ymn,\eps) - I) =\\
=\Phpl(\ypl,\eps)-\Phmn(\ymn,\eps) - (2 \Phpl(\ymn,\eps) -
\Phpl(\yptom)-\Phmn(\yptom)) =\\
=\Phpl(\ypl)-\Phmn(\ypl) - (2 \Phpl(\ypl) -
\Phpl(\yptom)-\Phmn(\yptom))+O(\eps) =\\
= (- \Phmn(\ypl) + \Phmn(\yptom)) + (- \Phpl(\ypl) + \Phpl(\yptom)) + O(\eps)=\\
=\eps^\lambda k_2(\ypl,\eps),
\end{multline}
where $k_2(\ypl,\eps)$ is a smooth function.

On the other hand, by definition of~$\ypl$,

\begin{equation}\label{eq-by-def-ypl}
\exp \left(\frac{\Phpl(\ypl,\eps)}{\eps}\right)=-\frac{b}{\xi}.
\end{equation}
The minus sign is due to $\xi<0$ in the case we are considering. It will be
opposite in the other case (when trajectory jumps up).

Substituting $\Phpl(\ypl,\eps)-\Phmn(\ymn,\eps)$ into~\eqref{eq-eta-of-xi}
with the expression which follows from~\eqref{eq-Phpl-Phmn}, and using~\eqref{eq-by-def-ypl}, we obtain:

\begin{equation}\label{eq-etof-xi}
\eta(\xi) = - \frac{b^2}{\xi} \exp\left( \frac{-I+\eps^\lambda k(\ypl,\eps)}{\eps}\right),
\end{equation}
where $k(\ypl,\eps)$ is a smooth function.

Let us show that $\eta(\xi)$ behaves like $-\const\cdot \xi^{-1}$,
i.e. it is convex for negative~$\xi$.

Equation~\bref{eq-by-def-ypl} implies that 

\begin{equation}
	\Phpl(\ypl,\eps)=\eps(\ln b-\ln(-\xi)).
\end{equation}
The function $\Phpl(\ypl,\eps)$ is strictly~$\ypl$-monotonic due
to nondegeniricity condition~\eqref{eq-nondeg-nondeg}. Thus there exists an
inverse function. Denote it by~$z_\eps$:

\begin{equation}
	z_\eps(\Phpl(\ypl,\eps))\equiv\ypl.
\end{equation}
Then

\begin{equation}\label{eq-ypl-via-xi}
	\ypl = z_\eps(\eps(\ln b - \ln (-\xi))).
\end{equation}
We substitute~\bref{eq-ypl-via-xi} into \bref{eq-etof-xi} and take the derivative:

\begin{gather*}
	\eta'(\xi) = b^2 e^{-I/\eps} \frac{1}{\xi^2} e^{(\eps^{\lambda-1} k(\ypl,\eps))}(k'z'_\eps \eps^\lambda +1),\\
	\ln \eta'(\xi) = 2 \ln b - \fr{I}{\eps} - 2 \ln (-\xi) + \eps^{\lambda-1} k(\ypl,\eps) + \ln (k'z'_\eps \eps^\lambda +1),\\
	\dd{}{\xi} \ln \eta'(\xi) = - \frac{2}{\xi} - \eps^\lambda k'z' \frac{1}{\xi} - \frac{\eps^{\lambda+1}}{\xi} \frac{k'' (z')^2 + k' z''}{1+k'z'\eps^\lambda} > \frac{1}{-\xi}.
\end{gather*}
Obviously,

\begin{equation}
	\Phpl(\ypl) > \frac{I}{3}.
\end{equation}
By \bref{eq-by-def-ypl}, we have

\begin{equation}
	-\xi = b \exp \frac{-\Phpl(\ypl,\eps)}{\eps}
\end{equation}
Hence,
\begin{equation}
	\dd{}{\xi}\log \eta'(\xi) > -\frac{1}{\xi} > \frac{1}{b} \exp \frac{I}{3 \eps}
\end{equation}
The transition from the normalizing charts back to the initial charts may only
multiply the derivative by a bounded function and will not considerably affect
the exponential estimate we just obtained:
\begin{equation}
	\dd{}{u} \ln (\Ppm)'>  \frac{1}{b} \exp \frac{I}{4 \eps}
\end{equation}
The chain rule and the estimate from assertion 2 of proposition~\ref{prop-q-est} imply that
\begin{multline}
	\dd{}{u} \ln (\Ppm)'\circ Q_\eps(u)=\dd{}{x} (\ln (\Ppm)')\cdot
	Q_\eps'(u) >\\
	> \frac{1}{b} \exp \left(\frac{I+o(1)}{4
	\eps}\right) > \frac{1}{b} \exp \left( \frac{I}{5\eps} \right).
\end{multline}
For trajectories that jump up, the corresponding estimate takes the form:

\begin{equation}
	\dd{}{u} \ln (\Ppm)'\circ Q_\eps(u)<-\frac{1}{\xi}< -\frac{1}{b} \exp \left( \frac{I}{5\eps} \right).
\end{equation}

\end{proof}

\subsection{Monotonicity: proof of lemma~\ref{lem-monot}}\label{ssec-monot}
In this section we prove lemma~\ref{lem-monot}. Let us first ensure that the
following assertions are fulfilled:

\begin{enumerate}
		\item	$\dde(x-y)(\bar C_\eps) \to \infty\text{ as } \eps\to 0^+$ for any
			choice of $\bar C_\eps=\Aepmb;\ \Eepmb$

	\item	The equation~$(y-x)(\Eemnb)= 2\pi n$ has root~$\eps=\eps_n$ for any~$n$, and $\eps_n=O(1/n)$. 
\end{enumerate}
	
Consider the cross-section~$\Gamma_0:=\{y=y_0\}$ for some~$y_0\in
[\delpl,\delmn]$. Consider the Poincar\'e map from~$\Ga_0$
to~$\Gamma=\{y=\pi\}$ in forward and reverse times: 

\begin{gather*}
	\Remn(x;y_0):=\lPeab{y_0}{-\pi}(x), \\
	\Repl(x;y_0):=\rPeab{y_0}{+\pi}(x).
\end{gather*}
We lift these maps from the circle $S^1_x$ to the universal cover $\mathbb
R^1_x$ continuosly in~$\eps$, and denote the result by $\bRepm$ ($y_0$ is
considered a fixed parameter).

The proof of the lemma is based on the following proposition:
\begin{proposition} 
	\label{prop-monot}
	One can find positive constants~$C^\pm$, such that for any fixed $x_0\in S^1$
	and any $\eps>0$ small enough, the following facts hold:
	\begin{enumerate}
		\item \label{en-semiorbit-monot} $\dd{\bRepm(x_0;y_0)}{\eps} \to \mp \infty$ as $\eps\to 0^+$,
		\item \label{en-semiorbit-order-p} $\bRepl(x_0;y_0)=\fr{C^{+}+O(\delpl^2)+O(\eps)}{\eps}$,
		\item \label{en-semiorbit-order-m} $\bRemn(x_0;y_0)=\fr{-C^{-}+O(\delmn^2)+O(\eps)}{\eps}$;
	\end{enumerate}
\end{proposition}
This proposition is proved in~\cite{GI}, (see the proof for
points~$\bar{d}_\eps^{-}$ and $\bar{A}^\pm_\eps$, pp.~45--46) for some
particular system, but the proof can be extended to our case verbatim.

\begin{proof}[Proof of Lemma~\ref{lem-monot}]
	Let us recall that
	\begin{gather*}
		\Aemn:=(\pepl,P_\eps(\pepl)),\quad
		\Aepl:=(\qepl,P_\eps(\qepl)),\\
		\Eemn:=(\pepl,\qemn),\quad \Eepl:=(\qepl,\pemn).
	\end{gather*}
	By definition, $\Depm=[\pepm,\qepm]$ (see subsection~\ref{ssec-shape} and
	fig.~\ref{fig-general-view} on page~\pageref{fig-general-view}),

	\begin{gather*}
		\pepl=\Remn(\spl;\delpl),\quad \qepl=\Remn(+\pi;\delpl),\\
		\pemn=\Repl(-\pi;\delmn),\quad \qemn=\Repl(\smn;\delmn).
	\end{gather*}
	It is obvious that	
	\begin{equation*}
	 P_\eps(\pepl)=\Repl(\spl;\delpl),\quad
	 P_\eps(\qepl)=\Repl(+\pi;\delpl).
	\end{equation*}
	Therefore,	

	\begin{equation}\label{eq-x-y-est}
		\begin{array}{rclll}
			(x-y)(\Aemn)&=&\Remn(\spl;\delpl)&-&\Repl(\spl;\delpl),\\
			(x-y)(\Aepl)&=&\Remn(+\pi;\delpl)&-&\Repl(+\pi;\delpl),\\
			(x-y)(\Eemn)&=&\Remn(\spl;\delpl)&-&\Repl(\smn;\delmn),\\
			(x-y)(\Eepl)&=&\Remn(+\pi;\delpl)&-&\Repl(-\pi;\delpl).
		\end{array}
	\end{equation}
Using~\bref{eq-x-y-est}, it is easy to show that assertion 1 of the lemma follows
from assertion~\ref{en-semiorbit-monot} of proposition~\ref{prop-monot}, and
assertion 2 of the lemma follows from the assertions~\ref{en-semiorbit-order-p}
and~\ref{en-semiorbit-order-m} of the same proposition.
\end{proof}

\section{Influence of the jump point: proof of technical propositions}\label{sec-technique}

\subsection{Dynamics near jump point}\label{ssec-jump-est-proof}
\subsubsection{Composition of the Poincar\'e maps}\label{sssec-poincare-comp}
In this section we will prove Theorem~\ref{thm-poincare-jump-est}. Let us show
that for some~$\la>0$, derivative of the Poincar\'e map
$\Reps=Q_\eps^{-1}=\lPeab{\alpl}{-\pi}$ can be written in the following form: 

\begin{equation}
	\ln
	\Reps'(x)=\frac{C(\delpl)+O(\eps^\lambda)}{\eps},
\end{equation}
where~$C(\delpl)<0$ is continuous and 
tends to 0 monotonically as ~$\delpl\to0^+$.

For some $\eps>0$, consider the two cross-sections 
\begin{gather}
	\Gaeps^1=\{(x,y) \mid y=\tpl-\eps^\nu\},
	\Gaeps^2=\{(x,y) \mid y=\tpl+\eps^\mu\},
\end{gather}
where~$\mu, \nu\in(0,1)$ are constants to be defined later.

Consider~$\Reps$ as a composition:

\begin{gather}
	\Reps: \Gapl \stackrel{\Reps^1}{\to} \Gaeps^2
	\stackrel{\Reps^2}{\to} \Gaeps^1 \stackrel{\Reps^3}{\to} \Gamma\\
	\label{eq-R-3-2-1}	\Reps=\Reps^3\circ \Reps^2 \circ \Reps^1
\end{gather}
Take the logarithmic derivative of~\eqref{eq-R-3-2-1}:

\begin{equation}\label{eq-lnRprim-decomp}
	\ln \Reps'=\ln (\Reps^3)'\circ \Reps^2\circ \Reps^1+\ln (\Reps^2)'\circ \Reps^1
	+\ln (\Reps^1)'.
\end{equation}
Lemma~\ref{lem-robust} implies that the second term of the sum
is~$O(\eps^{\min(\mu,\nu)-1})$. We will prove that

\begin{gather}
	|\ln (\Reps^3)'|<C\eps^{\nu-1},\\
	\ln (\Reps^1)'=\fr{C(\delpl)+O(\eps^\mu)}{\eps},\quad
	C(\delpl)<0.
\end{gather}
Taking~$\la=\min(\mu,\nu)$ and using these estimates
in~\bref{eq-lnRprim-decomp}, we obtain the the desired estimate and thus
complete the proof of theorem~\ref{thm-poincare-jump-est}.

\subsubsection{Dynamics far from the true slow curve}
The following proposition generalizes theorem~\ref{thm-nonlin-trans}. It allows
us to estimate the derivative of the Poincar\'e map from some given cross-section
to the the cross-section which slowly approaches the jump point.

Let us move the origin to~$\Gpl=(\spl,\tpl)$. Then
cross-section~$\Gamma=\{y=-\pi\}$ becomes $\{y=a\}$ for some $a>0$ (which can be
chosen arbitrary by an appropriate coordinate change).

Define the map:

\begin{equation}
	\Qeps^3=\inv{\Reps^3}=\rPeab{a}{-\eps^\nu}: \Ga\to\Gaeps^1.
\end{equation}

\begin{proposition}\label{prop-Reps-estimate}
	The derivative of~$\Qeps^3$ can be estimated as follows: 
	\begin{equation}\label{eq-Reps-est}
		|\ln (\Qeps^3)'|=|\ln (\Reps^3)'|< C \eps^{\nu-1}.
	\end{equation}
\end{proposition}	
Proposition~\ref{prop-Reps-estimate} is proved in subsection~\ref{ssec-Reps-estimate-proof}.

\subsubsection{True slow curve near the jump point}

\begin{proposition}\label{prop-sf-near-fold-est}
Assume that the conditions of Proposition~\ref{prop-Reps-estimate} are satisfied
for any~$u\in\Gapl$ in some neighborhood of the slow curve~$M^+$. Then, the
following equality holds: 
	\begin{equation}\label{eq-q-final-est}
		\ln (\Reps^1)'(u)=\fr{C(\delpl)+O(\eps^\mu)}{\eps},
	\end{equation}
	where~$C(\delpl)<0$ is continuous and tends to 0
	monotonically as~$\delpl\to0$.
\end{proposition}

To prove this proposition we need a description of the asymptotics of the true slow curve near
the jump point:

\begin{theorem}[\cite{MischenkoRozov}, p. 119]
	\label{thm-mr-tsc-est}
	There exists a~$\mu\in(0,1/3)$ such that the function defining true slow
	curve, $x=s(y,\eps)$, admits the following representation for~$y\in
	[\eps^\mu,\delpl]$ :

	\begin{equation}\label{eq-s-MR-est}
		s(y,\eps) = s(y) + O(\eps^{2/3-\mu/2}),
	\end{equation}
where~$s(y)$ defines the slow curve.
\end{theorem}
This theorem can be deduced by applying a trivial coordinate change to
representation (16.10) from the cited work.

\begin{remark}\label{rem-mu}
	In the definition of the cross-section~$\Gaeps^2$ we set $\mu$ equal to the value
	given by theorem~\ref{thm-mr-tsc-est}.
\end{remark}
\begin{proof}[Proof of the Proposition~\ref{prop-sf-near-fold-est}] 

	Fix a cross-section~$\Ga'=\{y=\delta'\}$ for~$\delta' \ll
	\delpl$. Assume that~$\eps$ is small enough to assure~$\delta'>\eps^\mu$.  We
	represent the map~$\Reps^1$ as a composition:

	\begin{equation}
		\Reps^1: \Gapl\stackrel{H^1_\eps}{\to} \Ga'
		\stackrel{H^2_\eps}{\to} \Gaeps^2.
	\end{equation}
	Suppose that the trajectory passing through the point~$(u,\delpl)$ is the 
	graph of a function~$x=x(y,\eps)$.  Theorem~\ref{thm-normnearss} implies
	that~$H^1_\eps$  is a linear contraction with an exponential
	rate of order~$O(e^{-C/\eps})$ in normalizing charts. Fix a segment of~$\Gapl$
	that intersects the slow curve~$M^+$.  Due to the exponential contraction, any
	trajectory that crosses this segment is quickly attracted to
	true slow curve in reverse time. We have:

	\begin{equation}
		|x(\delta',\eps)-s^+(\delta',\eps)|\le \const e^{-C/\eps}
	\end{equation}
	With an appropriate choice of $\delta'$, we can move~$\Gamma'$ arbitrarily
	close to $\Gaeps^2$ for $\eps$ small enough. Thefore, the trajectory spends much
	more time during the transition from $\Gapl$ to $\Gamma'$ than during the transition
	from $\Gamma'$ to $\Gaeps^2$.  Hence, lemma~\ref{lem-robust} implies that for
	any  $y\in[\eps^\mu,\delta']$ the following estimate holds: 

	\begin{equation}\label{eq-xs-near-jump-est}
		|x(y,\eps)-s^+(y,\eps)|\le \const
		e^{-C/\eps}
	\end{equation}
	In other words, the corresponding trajectory on the segment from~$\Gamma'$ to
	$\Gaeps^2$ moves exponentially close to the true slow curve. 

	Variational equations imply: 

	\begin{multline}\label{eq-q-prim-est}
		\ln (\Reps^1)'(u)=\ln (H^1_\eps)'(u)+\ln (H^2_\eps)'\circ
		H^1_\eps(u)=\\=
		\fr{-c_3(\delta',\delpl)}{\eps}+\frac{1}{\eps}\int_{\delta'}^{\eps^\mu}
		f'(x(y,\eps),y,\eps)dy+O(1),
	\end{multline}
	where~$c_3(\delta',\delpl)>0$.

	Due to~\bref{eq-xs-near-jump-est}, one can replace $x(y,\eps)$ by
	$s^+(y,\eps)$ in the second term of~\eqref{eq-q-prim-est}. Moreover,
	using representation~\bref{eq-s-MR-est}, we obtain:

	\begin{multline}
		\label{eq-q-two-est}
		\ln (H_\eps^2)'=\fr{1}{\eps}\int_{\delta'}^{\eps^\mu} f'_x(x(y,\eps),y,\eps)dy=
		\fr{1}{\eps}\int_{\delta'}^{\eps^\mu}f'_x(s^+(y,\eps),y,\eps) dy + O(1)
		=\\
		= \fr{1}{\eps}\left[\int_{\delta'}^{\eps^\mu} f'_x(s^+(y),y,0) dy
		+O(\eps^{2/3-\mu/2})\right]=\fr{1}{\eps}\left[ \int_{\delta'}^0
		f'_x(s^+(y),y,0)dy+O(\eps^{\mu}) \right]=\\=\fr{-c_4(\delta')+O(\eps^{\mu})}{\eps}
	\end{multline}
	Denoting~$C=-(c_3+c_4)$ and substituting~\bref{eq-q-two-est}
	into~\bref{eq-q-prim-est}, we obtain~\bref{eq-q-final-est}. The number $C$
	is~$\delpl$-monotonic, because expressions under the integral sign
	in~\eqref{eq-q-prim-est} and~\eqref{eq-q-two-est} are positive. 
\end{proof}

\subsection{Distortion lemma: proof of proposition~\ref{prop-Reps-estimate}}\label{ssec-Reps-estimate-proof}

In this section we prove proposition~\ref{prop-Reps-estimate}. Instead of
considering the Poincar\'e map from one vertical cross-section to another, we
consider Poincar\'e map from the horizontal cross-section~$x=0$ to itself. We will
iterate this Poincar\'e map and estimate the derivative of these iterations using
the~\emph{Distortion Lemma} due to Denjoy and Schwartz. To apply this lemma we
need some additional estimates obtained from the variational equations. 

\subsubsection{Vertical Poincar\'e map}\label{ssec-vertical}
\begin{proof}[Proof of the proposition~\ref{prop-Reps-estimate}]
	Along with the system~\bref{eq-main} we will consider two auxilliary systems
	with the same phase portraits for any given~$\eps$:

	\begin{equation}\label{eq-sf-norm-y}
		\dt x = v(x,y,\eps),\quad
		\dt y = \eps;
	\end{equation}
	and
	\begin{equation} \label{eq-sf-norm-x}
		\dt x= 1,\quad
		\dt y=\eps w(x,y,\eps),
	\end{equation}
	where 

	\begin{equation}
		v(x,y,\eps)=\frac{f(x,y,\eps)}{g(x,y,\eps)},\quad
		w(x,y,\eps)=\frac{g(x,y,\eps)}{f(x,y,\eps)}.
	\end{equation}
	According to remark~\ref{rem-dividing-out-g}, without loss of generality we
	can assume that $g=1$ and then $v=f$, $w=1/f$.

	Nondegenerecity conditions~\bref{eq-nondeg-main-wlog} imply that for~$a<y<0$
	and for some positive constants~$c_v, C_v, c_w, C_w$ the following equivalent
	estimates hold:	

	\begin{equation}\label{eq-vw-est}
		0<-c_v (y+O(\eps))<v(x,y,\eps)<C_v\quad\Leftrightarrow\quad
		0<c_w<w(x,y,\eps)<-\frac{C_w}{y+O(\eps)}
	\end{equation}
	For the sake of simplicity, we will rescale the coordinate $x\mapsto 2\pi x$ so
	that $x$ becomes a coordinate modulo $\bbZ$. 

	Denote the line $\{x=0\}=\{x=1\}$ by~$\Sigma$, and the Poincar\'e
	map~$\Sigma\to\Sigma$ by~$\psi$ (see fig.~\ref{fig-general-view-distortion}).
	Recall that we are considering the following cross-sections:	

	\begin{gather*}
		\Ga=\{(x,y)\mid y=a\}\\
		\Gaeps^1=\{(x,y) \mid y=-\eps^\nu\}.
	\end{gather*}
	For brevity, we will write $\Gaeps$ instead of~$\Gaeps^1$.

	Let us define a correspondence map~$\gamma$ from the cross-section~$\Ga$ to
	the semi-interval~$J_0\subset\Sigma$ in forward time (see
	fig.~\ref{fig-general-view-distortion}). Note
	that~$\gamma$ is not countinous at~$0$, so in the following analysis we will
	replace the circle~$\Ga$ by the semi-interval~$[0,1)$, where~$\gamma$ is
	continious. Consequently, let us denote:

	$$J_{k+1}:=\psi(J_k),\quad \psi|_{J_k}=:\psi_k, \quad k\ge 1$$
	Obviously, the semi-intervals~$J_k$ do not intersect each other and the right
	end of the $k$-th interval coincides with the left end of $(k+1)$-th
	interval.  Define~$N=N(\eps)$ in such way that~$J_N$ intersects~$\Gaeps$.
	Denote by $\tau$ the intersection point of $\Gaeps$ and the trajectory
	passing thought $(0,a)$. 

	\begin{figure}[t]
		\psfrag{Ga}{$\Ga$}
		\psfrag{J0}{$J_0$}
		\psfrag{J1}{$J_1$}
		\psfrag{J2}{$J_2$}
		\psfrag{J3}{$J_3$}
		\psfrag{JN}{$J_N$}
		\psfrag{Ge}{$\Gamma_\eps$}
		\psfrag{tau}{$\tau$}
		\psfrag{a}{$a$}
		\psfrag{be2}{$-\frac{\eps^\nu}{2}$}
		\psfrag{be}{$-\eps^\nu$}
		\psfrag{x}{$x$}
		\psfrag{y}{$y$}
		\psfrag{0}{$0$}
		\psfrag{1}{$1$}
		\psfrag{f}{$\psi$}
		\psfrag{Peps}{$Q^3_\eps$}
		\psfrag{T}{$\Gpl$}

		\begin{center}
			\includegraphics[scale=0.8]{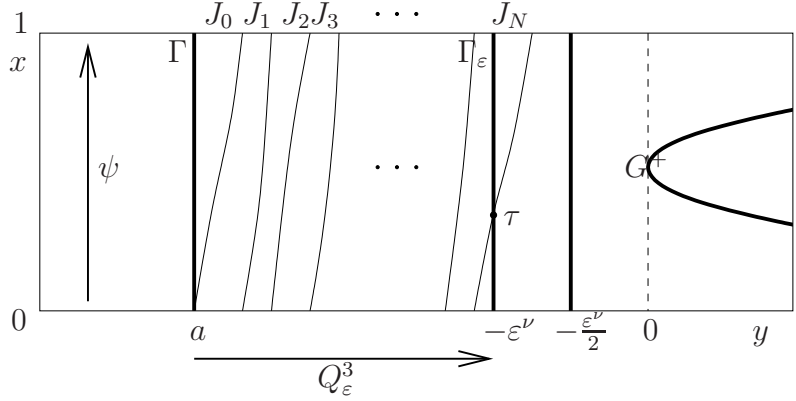}
			\caption{Correspondence maps}\label{fig-general-view-distortion}
		\end{center}
	\end{figure}
	Let us define~$\gaeps: \Gaeps\to J_N$ in the following way. On 
	the interval~$\rarc{0}{\tau}$ (\emph{resp.}, $\rarc{\tau}{1}$) it coincides with the correspondence
	map for the phase flow of~\bref{eq-main} in the reverse (\emph{resp.}, forward)
	time. Defined this way, ~$\gaeps$ will be discontinious in~$\tau$, but
	its inverse~$\gaeps^{-1}$ continiuosly projects~$J_N$ to the
	circle~$\Gaeps$ along the phase curves.

	Now~$Q^3_\eps$ may be represented by the following composition:

	\begin{gather*}
		Q^3_\eps: \Gamma\stackrel{\gamma}{\to}
		J_0\stackrel{\psi_0}{\to}\cdots\stackrel{\psi_{N-1}}{\to} J_N
		\stackrel{\gaeps^{-1}}{\to} \Gaeps\\
		Q^3_\eps=\gaeps^{-1}\circ \psi_{N-1}\circ \cdots \circ \psi_0 \circ
		\gamma=\gaeps^{-1}\circ F_{N-1}\circ \gamma	
	\end{gather*}
	The chain rule implies:
	\begin{equation}\label{eq-lnPprim-decomp}
		|\ln \Qeps'| \le |\max_{J_N}\ln |(\gaeps^{-1})'|| + |\max_{J_0}\ln
		F'_{N-1}| + |\max_{\Gamma}\ln |\gamma'||.
	\end{equation}
	We will obtain the following estimates for the terms of this sum.

	\begin{gather}
		|\ln |\gaeps'|| < C \eps^{\nu-1}\label{eq-lnPprim-gaeps}\\
		|\ln F_{N-1}'| \le (\nu-1) \ln \eps + \const + o(1)\label{eq-lnPprim-middle}\\
		|\ln |\gamma'|| \le \ln \frac{1}{\eps}\label{eq-lnPprim-gamma}
	\end{gather}
	These estimates are obtained below and they justify~\eqref{eq-Reps-est}.
\end{proof}
\begin{remark}
	Inequality~\bref{eq-lnPprim-decomp} can be applied to every point of~$\Gamma$
	excluding~$0$, because $\gamma$ and $\gaeps$ are discontinuous in~$0$
	and~$\tau$ respectively. However, we can still obtain necessary inequalities
	on the whole~$\Gamma$. To this end, we have to define another
	cross-section~$\Sigma$ (e.g.~$\Sigma=\{x=\frac{1}{2}\}$) and use the same
	arguments for this new cross-section to obtain the inequality for $x=0$.
\end{remark}
\subsubsection{Application of the Distortion Lemma}
In this section we will obtain~\bref{eq-lnPprim-middle}. First, we give the
statement of the~\emph{Distortion Lemma}.
\begin{definition}
For any diffeomorphism~$\psi: A \to B$ define its \emph{distortion rate} as:

	\begin{equation}
		\kappa(\psi)=\ln \frac{\max_A \psi'}{\min_A \psi'}=\max_{x,y\in A} \ln
		\frac{\psi'(x)}{\psi'(y)}.
	\end{equation}
\end{definition}
\begin{lemma}[Distortion Lemma (\cite{Denjoy}, \cite{Schwartz})]\label{lem-dist}
	Consider a sequence of arbitrary intervals and their orientation-preserving diffeomorphisms:	
	$$
	J_0\stackrel{\psi_0}{\to} J_1 \stackrel{\psi_1}{\to}\cdots
	\stackrel{\psi_{N-1}}{\to} J_N.
	$$
	Then the following estimate holds for the
	composition~$F_{N-1}:=\psi_{N-1}\circ\cdots\circ \psi_{0}$ of these maps: 
	\begin{equation}\label{distortion-lem-est}	
		\kappa(F'_{N-1}) \le \max_{i} \max_{x\in J_i}\cdot
		\left|\frac{\psi_i''(x)}{\psi_i'(x)}\right| \sum_{k=0}^{N-1} |J_k|.
	\end{equation}
\end{lemma}
This lemma easily follows from the subadditivity (under the composition) of
the distortion rate and the mean value theorem, which implies that:
\begin{equation}
	\kappa(\psi)=\max_{x,y\in A}(\ln \psi'(x)-\ln \psi'(y))\le \max_{z\in
	A}\left|\fr{\psi''(z)}{\psi'(z)}\right| \cdot |A|.
\end{equation}
Our strategy is to apply the Distortion Lemma to composition $F_{N-1}$ of the
iterations of the vertical Poincar\'e map~$\psi$ (see subsection~\ref{ssec-vertical}).
To this end, we have to
estimate the sum of lengths for~$\{J_k\}_{k=0}^{N-1}$ and the 
distortion rate of~$\psi$. In the rest of this subsection we consider 
system~\bref{eq-sf-norm-x}.

By construction, the intervals~$J_k$ do not intersect each other. Therefore, to
estimate the sum of their lengths it is sufficient to control the last
interval~$J_N:=[x_N,x_{N+1}]$.

\begin{proposition}\label{prop-lastJ}
	For any~$\nu<1/2$ and~$\eps$ small enough, we have:
	\begin{gather}\label{eq-J-N-est}
		|J_N|\le \eps^\nu\\
		x_{N+1}<-\fr{1}{2}\eps^{\nu}
	\end{gather}
\end{proposition}
\begin{proof}
	Let us denote the projections of the phase space to~$\Si$ along the phase
	curves in forward ~$(+)$ and reverse~$(-)$ time as~$T^\pm$. Obviosly, 
	\begin{multline}
		J_N=[T^-(\tau,-\eps^\nu),T^+(\tau,-\eps^\nu))\subset\\
		\subset[T^-(1,-\eps^\nu),T^+(0,-\eps^\nu)]=[\psi^{-1}(-\eps^\nu),\psi(-\eps^\nu)].
	\end{multline}
	Let us consider the trajectory that passes through the point~$(0,-\eps^\nu)$.
	The second equation of the system~\bref{eq-sf-norm-x} is:

	\begin{equation}
		\dt y=\eps w(x,y,\eps)
	\end{equation}
By~\eqref{eq-vw-est}, the right-hand side of this equation is positive and can be estimated from above by
	$-\eps\frac{C_w}{y+O(\eps)}$ (recall that in the domain under
	considertion~$y<0$). Solving the equation

	\begin{equation}
		\dt y_* = -\eps \frac{C_w}{y_*+O(\eps)},
	\end{equation}
	we find that

	\begin{equation}
		y_*(t;y_0)=-\sqrt{y_0^2-2\eps C_w t+O(\eps)},\quad
		y_*(0;y_0)=y_0<0.
	\end{equation}
	Therefore, 
	\begin{multline}
		T^+(0,-\eps^\nu)=
		y(1;-\eps^\nu)
		\le y_*(1;-\eps^\nu)=
		-\sqrt{\eps^{2\nu} + O(\eps)} = \\
		= - \eps^\nu \sqrt{1+O(\eps^{1-2\nu})}<
		-\frac{1}{2} \eps^\nu
	\end{multline}
	The last inequality holds for~$\eps$ small enough if $\nu<1/2$.

	Similar arguments show that~$T^-(0,-\eps^\nu)>-\frac{3}{2}\eps^\nu$, which
	implies the proposition.
\end{proof}
\begin{remark}
	Obviosly,~\bref{eq-vw-est} implies the following estimates for
	some positive constants
	$C_{J_N}$, $\ C_{J_0}$ and $c_{J_0}$:

	\begin{equation}\label{eq-J-misc-est}
		|J_N| \ge C_{J_N} \eps,\quad c_{J_0} \eps \le |J_0| \le C_{J_0} \eps.
	\end{equation}
\end{remark}
\begin{proposition}
	For any~$\nu<1/2$ the derivative of the Poincar\'e map~$\psi: \Si\to\Si$ admits
	the following estimate in the domain $a<y<-\eps^\nu$:
	\begin{gather}\label{eq-fprim-est}
		\psi' \ge \frac{1}{2},\\
		\psi'' \le C\eps^{1-4\nu} \label{eq-frpimprim-est}.
	\end{gather}
\end{proposition}
\begin{proof}
	Proposition~\ref{prop-lastJ} implies that a trajectory starting from the
	domain~$\{a\le y \le -\eps^\nu\}$ does not leave the domain 
	$\{a-O(\eps)\le y \le -\eps^\nu/2\}$ until it intersects~$\Si$.

	Let $y=y(t;y_0)$~ be the~$y$-coordinate of the solution of the
	system~\bref{eq-sf-norm-x} with the initial conditions~$x(0)=0,
	y(0)=y_0$. Obviously,

	\begin{equation}
		\psi(y_0)=y(1;y_0).
	\end{equation}
	Taking the derivative with respect to~$y_0$, we obtain:

	\begin{equation}
		\psi'(y_0)=\padi{y(1;y_0)}{y_0}=:\yy(1;y_0).
	\end{equation}
	The variational equation for~$\yy$ has the form:

	\begin{equation}\label{eq-var-yy}
		\dt{y}'_{y_0}  = \eps \padi{w}{y} \yy,\quad \yy(0;y_0)=1.
	\end{equation}
	Therefore,

	\begin{equation}\label{eq-yy-expint}
		\yy(t;y_0) = \exp\int_0^t 
		\eps\padi{w}{y}(x(\xi;y_0),y(\xi;y_0),\eps)\, d\xi.
	\end{equation}
	For~$a<y<-\frac{1}{2}\eps^\nu$ the derivative~$\padi{w}{y}$ 
	can be estimated as follows:

	\begin{equation}\label{eq-w2-est}
		\left|\padi{w}{y} \right|=\left|\padi{}{y}
		\frac{g(x,y,\eps)}{f(x,y,\eps)}\right| =
		\left| \frac{g'_y f-f'_y g}{f^2} \right| \le
		\frac{C}{(y+O(\eps))^2} \le
		\frac{C_1}{\eps^{2\nu}}.
	\end{equation}
	Therefore, for~$\nu<\frac{1}{2}$ and~$\eps$ small enough,

	\begin{equation}
		\psi'(y_0) \ge \exp\left(-\eps
		\frac{C_1}{\eps^{2\nu}}\right)=
		\exp(-C_1 \eps^{1-2\nu})>\frac{1}{2},
	\end{equation}
	which proves~\bref{eq-fprim-est}.

	To obtain~\bref{eq-frpimprim-est} we take the derivative
	of~\bref{eq-yy-expint} with respect to~$y_0$ for~$t=1$:

	\begin{multline}
		\psi''(y_0)=\padi{ }{y_0} \yy(1;y_0) 
		=\padi{}{y_0} \exp\int_0^1 
		\eps \padi{w}{y}(x(\xi;y_0),y(\xi;y_0),\eps)\, d\xi =\\
		= \eps \psi'(y_0) \int_0^1 \padi{^2 w}{y^2} \yy(\xi;y_0)\,
		d\xi \le \eps (\max_{\xi\in[0,1]} \yy(\xi;y_0))^2 \frac{C_2}{\eps^{4\nu}}
		\le\\
		\le C \eps^{1-4\nu} \exp(C \eps^{1-2\nu})\le C 
		\eps^{1-4\nu},
	\end{multline}
	where~$\padi{^2 w}{y^2}$ was estimated from above
	by~$\frac{C_2}{\eps^{4\nu}}$ (the justification is similar to~\bref{eq-w2-est}),
	and $\yy$ was estimated using~\bref{eq-var-yy} and~\bref{eq-w2-est}.
\end{proof}
\begin{corollary}
	For~$\nu<1/4$ and~$\eps\to 0$ we have:

	\begin{equation}
		\frac{\psi''}{\psi'} \le C \eps^{1-4\nu} \to 0.
	\end{equation}
\end{corollary}

Therefore, the Distortion Lemma gives the following estimate for~$F_{N-1}$
(for brevity, we omit the index $N-1$ below):

\begin{equation}\label{eq-dist-lemma-F-est}
	\ln \frac{\max_{J_0} F'}{\min_{J_0} F'}
	\le C \max_{[a,-\eps^\nu]} \frac{\psi''}{\psi'} = o(1).
\end{equation}
It implies that

\begin{equation}\label{eq-dist-lemma-Fmaxmin}
	\ln\max_{J_0} F'=\ln\min_{J_0} F'+o(1).
\end{equation}
At the same time, the mean value theorem implies that for any~$\eps>0$ small enough
there exists some point~$y_0\in J_0$, such that

$$F'(y_0)=\frac{|J_N|}{|J_0|}.$$
According to the estimates~\bref{eq-J-N-est} and~\bref{eq-J-misc-est},
we have: 
\begin{equation}\label{eq-minmax-F-est}
	\begin{array}{rcccl}
		\min F' & \le & F'(y_0)& \le& C_1 \eps^{\nu-1},\\
		\max F' &\ge & F'(y_0) & \ge & C_2.
	\end{array}
\end{equation}
Taking the lograhithm of~\bref{eq-minmax-F-est} and using~\bref{eq-dist-lemma-Fmaxmin}, we have:

\begin{equation}
	\begin{array}{rcl}
		\ln \max F'&\le& \ln C_1\eps^{\nu-1}+o(1),\\
		\ln \min F'&\ge& \const+o(1).
	\end{array}
\end{equation}
Therefore,~\bref{eq-lnPprim-middle} is justified.

\subsubsection{The projection to the horizontal cross-section}
In this subsection inequalities~\bref{eq-lnPprim-gaeps}
and~\bref{eq-lnPprim-gamma} are proved. Consider the
system~\bref{eq-sf-norm-y}.

\begin{figure}[t]
	\psfrag{tga}{$\tgam$}
	\psfrag{a}{$a$}
	\psfrag{gax}{$\gamma(x)$}
	\psfrag{tx}{$\tilde{x}$}
	\psfrag{gax0}{$\gamma(x_0)$}
	\psfrag{x0}{$x_0$}
	\psfrag{R}{$R$}
	\psfrag{Ga}{$\Ga$}
	\psfrag{tGa}{$\Ga'$}
	\psfrag{y}{$y$}
	\psfrag{x}{$x$}
	\begin{center}
		\includegraphics[scale=0.6]{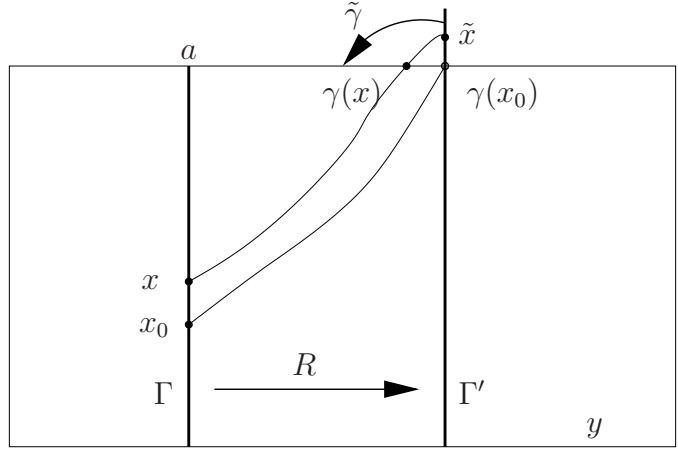}
		\caption{Estimate of the derivative~$\gamma$}\label{fig-gamma-estimate}
	\end{center}
\end{figure}
\begin{proposition}
	For the map~$\gamma: \Gamma \to J_0$ the following estimate holds: 
	\begin{equation}
		|\ln |\gamma'|| \le \ln \frac{1}{\eps}+O(1).
	\end{equation}
\end{proposition}
\begin{proof}
	Let us fix~$x_0$ and define~$y_0:=\gamma(x_0)$. We present~$\gamma(x)$ as a
	composition (see fig.~\ref{fig-gamma-estimate}):

	\begin{equation}
		\gamma: \Gamma \stackrel{R}{\to} \Gamma'
		\stackrel{\tgam}{\to} \Si,
	\end{equation}
	where~$\Gamma'=\{y=y_0\}$ is the shifted vertical cross-section,~$R$ is the
	correspondence map from~$\Gamma$ to $\Gamma'$, and~$\tgam$ is the correspondence
	map from some neighborhood of the point~$x=0$ on~$\Gamma'$ to some
	neighborhood of $y=y_0$ on~$\Si$, which is defined in the following way:

	\begin{equation}
		\tgam(\pm \tix)=T^{\mp} (\tix,y_0),
	\end{equation}
	Recall that~$T^{\pm}$ are the projections of the phase space to~$\Si$ along
	the trajectories of the system in forward~$(+)$ and reverse $(-)$ time.

	It is easy to see that~$\tgam'(1)=-\eps w(0,\gamma(x_0),\eps)$, because in
	a small neighborhood of~$(1,y_0)$ the map $\tgam$ is close to the linear
	projection from~$\Gamma'$ to~$\Si$ along the vector $(1,-\eps w(1,y_0,\eps))$.

	The function~$w$ is bounded away from zero and infinity in some neighborhood of
	the cross-section~$\Ga$. Therefore~$|\tgam'|$ is of order~$\eps$.

	Note, that the distance between~$\Gamma$ and~$\Gamma'$ is of order~$\eps$,
	and therefore the time of the transition between these cross-sections along the phase curves
	of the system~\bref{eq-sf-norm-y} is bounded from above. Hence,
	the derivative of~$R$ is bounded from above (what follows from the variational
	equations). Therefore, estimates for~$\tgam'$ and $\gamma'$ coincide.
\end{proof}
\begin{proposition}
	For every point $x_0\ne\tau$, the following estimate holds: 

	\begin{equation}\label{eq-gaeps-eps-nu-est}
		|\ln |\gaeps'(x_0)|| < C \eps^{\nu-1}
	\end{equation}
\end{proposition}
\begin{proof}
	Like in the previous proposition, we fix~$x_0$ and represent~$\gaeps$ as
	a composition:

	\begin{equation}
		\gaeps: \Gaeps\stackrel{\Reps}{\to} \Gaeps'
		\stackrel{\tgam_\eps}{\to} J_N,
	\end{equation}
	where~$\Gaeps'=\{y=\gaeps(x_0)\}$.

	According to proposition~\ref{prop-lastJ}, all the trajectories starting
	at~$\Gaeps$ stay in the strip $\{-\frac{3}{2}\eps^\nu < y <
	-\frac{1}{2}\eps^\nu\}$ until they cross~$\Si$ in forward or reverse time.
	Therefore, the trajectory of~\bref{eq-sf-norm-y} spends time~$t_*=C_1
	\eps^{\nu-1}$ before it intersects~$\Si$. Due to the smoothness
	of~$v(x,y,\eps)$ on the whole torus, the variational equations imply that 
	\begin{equation}\label{eq-R-eps-est}
		|\ln R'| \le C_2 t_* \le C \eps^{\nu-1}.
	\end{equation}
	Applying arguments from the previous proposition to~$\tgam_\eps'$, we obtain:

	\begin{equation}\label{eq-tgam-ln-est}
		\tgam_\eps'(0)=-\eps w(1,y_0,\eps),\quad |\ln \tgam_\eps'(0)| <
		(\nu-1) \ln \eps+O(1) \ll \eps^{\nu-1}.
	\end{equation}
Provided~\eqref{eq-R-eps-est} and~\eqref{eq-tgam-ln-est}, the chain rule
justifies~\eqref{eq-gaeps-eps-nu-est}.
\end{proof}
\begin{remark} We can put arbitrary~$\nu$ from the interval $(0,1/4)$ into the
	definition of the cross-section $\Gaeps=\Gaeps^1$.
\end{remark}



\end{document}